\documentclass[12pt,reqno]{amsart}
\usepackage{graphicx}
\usepackage{amsmath,amssymb,amsthm}
\usepackage[unicode,breaklinks=true,colorlinks=true]{hyperref}
\usepackage[dvipsnames]{xcolor}
\usepackage[normalem]{ulem} 
\usepackage{comment}
\usepackage{tikz}
\usepackage{psfrag}
\usepackage{graphicx}
\usepackage{epsfig} 

\usepackage{float}
\usetikzlibrary{patterns}

\usepackage[top=1in, bottom=1in, left=1in, right=1in, marginparwidth=1in, marginparsep=0.1in]{geometry}

\numberwithin{equation}{section}
\newtheorem{theorem}{Theorem}[section]

\newtheorem{lemma}[theorem]{Lemma}
 
\newtheorem{corollary}[theorem]{Corollary}

\theoremstyle{remark}
\newtheorem{remark}[theorem]{Remark}

\definecolor{darkblue}{rgb}{0,0,0.7}

\usepackage{marginnote} 

\newcommand{\interp}{\mathcal{K}}

\newcommand{\al}{\alpha}

\newcommand{\ga}{{\gamma}}

\newcommand{\la}{\lambda}

\newcommand{\Om}{{\Omega}}
\newcommand{\om}{{\omega}}
\newcommand{\tom}{{\tilde \omega}}
\newcommand{\tu}{{\tilde u}}
\newcommand{\tpsi}{{\tilde \psi}}

\newcommand{\td}{\tilde}

\newcommand{\N}{{\mathbb N}}

\newcommand{\bk}{{\mathbf k}}

\newcommand{\nb}{{\nabla}}

\newcommand{\I}{\infty}

\newcommand{\supp}{\mathop{\mathrm{supp}}}

\newcommand{\donothing}[1]{{}}

\newcommand{\EQ}[1]{\begin{equation}\begin{split} #1 \end{split}\end{equation}}
\newcommand{\EQN}[1]{\begin{equation*}\begin{split} #1 \end{split}\end{equation*}}

\makeatletter
\newcommand{\xRightarrow}[2][]{\ext@arrow 0359\Rightarrowfill@{#1}{#2}}
\makeatother

\usepackage{accents}

\begin{document}
\title[Mobile data assimilation for 2D NS]{Convergence of a mobile data assimilation scheme for the 2D Navier-Stokes equations}
\author{Animikh Biswas, Zachary Bradshaw and Michael Jolly}
\date{\today}
\maketitle 

\begin{abstract}  
We introduce a localized version of the {nudging} data assimilation algorithm for the periodic 2D Navier-Stokes equations in which observations are confined {(i.e., localized)} to a window that moves across the entire domain along a predetermined path at a given {speed}. We prove that, if the {movement is fast} enough, then the algorithm perfectly synchronizes with  a reference solution. The analysis suggests an \textit{informed} scheme in which the subdomain moves according to a region where the error is dominant is optimal.   Numerical simulations are presented that compare the efficacy of movement that follows a regular pattern, one guided by the dominant error, and one that is random. 
\end{abstract}


\section{Introduction}

Data assimilation is concerned with recovering the fine scale activity of a dynamical system via coarse measurements of that system \cite{Daley1991, K}. In applications, exact knowledge of an initial state is often unavailable. However, sensors often continuously monitor activity at a coarse scale. This, for instance, is the case in atmospheric sciences where, since the launch of the first weather satellites in the 1960s, weather data has been collected nearly continuously in time. These measurements provide us with partial knowledge of the state of the system, e.g., the velocity vector field or temperature, on a  coarse spatial grid of points. In data assimilation, forecasting is achieved by supplementing a model with coarse measurements as opposed to a complete initial state.  Many other applications exist including, but not limited to, environmental sciences, systems biology and medicine \cite{Kost2}, imaging science, traffic control and urban planning, economics and finance and oil exploration \cite{ABN}. 
 
In \cite{AOT,AT}, Azouni, Olson and Titi introduced   a nudging based data assimilation algorithm which is mathematically rooted in earlier work on determining functionals \cite{FP,FT,JonesTiti,JonesTiti2}. The concept of nudging is, however, much older and was developed to address geophysical and control problems. An upside of the approach of \cite{AOT}  is its ease of implementation and amenability to rigorous analysis. In contrast, for more traditional approaches to data assimilation, e.g., the Bayesian and variational frameworks, issues of stability, accuracy and {catastrophic filter divergence} persist \cite{HM1, tmk2016-1, tmk2016-2}.
Since \cite{AOT}, the nudging scheme has received a great deal of attention from the fluids community. A partial list of references are  \cite{ANT,BOT,BCM,BM,CHL,FGHMMW,FMT,MTT}.

In the Azouni, Olson and Titi setup, the coarse grid on which data is collected must span the entire domain.  This may be costly or unrealistic in real world settings.  Therefore, it is desirable to develop a data assimilation algorithm which either does not require data to be continuously collected  across the full domain or requires observations across the full domain only at a very coarse scale (in time and space).
In \cite{BBJ1} we demonstrated that, within an error, a localized, stationary collection of observations taken from an \textit{observability region} can approximate a reference flow---i.e., synchronization occurs up to a non-zero error in contrast to \cite{AOT} where synchronization is exact. A numerically visible defect of the method is that the convergence rate of global synchronization, while exponential, is slower than  that of local synchronization on the observability region. The rapid local synchronization slows down global synchronization because the feedback operator is itself local (to the observability region). This means that, when the reference and approximating solutions are locally almost synchronized, the nudging becomes ineffective. If the observability window is moved to a region where the  reference  and approximating solutions are not locally almost synchronized, then the nudging is strengthened, at least until local synchronization occurs in the new region. This intuitive discussion suggests that building mobility into the localized nudging operator will result in improved synchronization compared to the immobile localized setup of \cite{BBJ1}. 
These insights are supported by numerical work on the Navier-Stokes equations \cite{BBJ1,LariosMobile}---see also \cite{LV}. In particular, Franz, Larios and Victor provide a detailed examination of a number of mobile paradigms, the so called ``bleeps, sweeps and creeps'' \cite{LariosMobile}, which outperform the static observer case (with the same number of measurements) in simulations.

In the present paper, we  rigorously show that a certain mobile data assimilation scheme \textit{exactly synchronizes} with a reference solution provided the observers are moving fast enough. We additionally study features of mobile data assimilation numerically.  Our numerical findings are: 1.~ increasing the frequency at which an observability window moves across the domain leads to faster convergence and 2.~choosing the observability window based on an even coarser decision protocol improves convergence compared to a pre-determined movement pattern. We shall refer to this as the informed scheme.  {The main motivations for local data assimilation are that fine-scale measurements may be more expensive to obtain and that collecting fine-scale data may be infeasible in parts of the domain (e.g., in the upper atmosphere or deep in the ocean). In our numerical implementation, the dominant region is identified by coarse-scale measurements only. Thus, the dominant scheme provides the possibility that one may not have to collect fine-scale data from such regions if they remain {\em inactive} until global synchronization has occurred at a high level of accuracy.}

In our simulations, we found that the dominant scheme synchronizes faster than {schemes where the subdomain moves in a regular pattern or moves randomly.} In some cases the motion of the observability window is discontinuous. Physically, this discontinuous movement is consistent with the deployment of different sets of observers consecutively in different regions as   chosen by a decision protocol. A contrived example where discontinuous motion makes sense in weather forecasting is as follows: If the protocol says observations should first be taken over California and, later, over New York, then drones stationed in California could first be deployed and, later, drones from New York could be used. {To be realistic, we also simulate the effect of delay in moving the observation domain}.

\subsection{Notation and preliminaries}

We consider the two-dimensional 
Navier-Stokes equations
(henceforth, 2D NSE) on  $\Om=[-L/2,L/2]^2$ with periodic boundary conditions. For $U\subset \Om$, we use $L^p(U)$ and $H^s(U)$ to denote the Lebesgue spaces and $L^2$ based Sobolev spaces respectively. If $U$ is omitted it is understood that $U=\Om$. We use $\langle \cdot,\cdot\rangle$ to denote the $L^2$ inner product. The Leray projection of $L^2$ onto mean zero divergence free functions in $L^2$ is denoted by  $\mathbb P$. Note that for periodic boundary conditions, $\mathbb P \Delta= \Delta \mathbb P$ and we therefore do not need to introduce the Stokes operator, cf.~\cite[p.~284]{AOT}. We use $|\cdot |$ to denote the absolute value of a vector or scalar and the 1D or 2D Lebesgue measure of a measurable set---the meaning will always be clear based on context. The characteristic function for the set $S$ is denoted $\chi_S$.

We make frequent use of the Poincar\'e and Ladyzhenskaya inequalities which respectively say:
For $u\in H^1(\Om)$, 
\EQ{\label{ineq.lady}
\| u\|_{L^2} \leq \la_1^{-1/2} \|\nb u\|_{L^2} \text{ and }\| u\|_{L^4}\leq C_L \| u\|_{L^2}^{1/2}\| \nb u\|_{L^2}^{1/2}.
}
The prefactors can be viewed as the optimal constants for which these estimates hold ($\la_1$ is the first eigenvalue of the Laplace operator). 

Data assimilation in the spirit of \cite{AOT} uses an interpolant operator to nudge an assimilating solution toward a reference flow. In the present paper, this will be based on a ``type 1'' interpolant $I_h$ defined using volume elements:
\[
I_h f (x)  = \sum_{i=1}^{M^2}\bigg(\chi_{S_i}(x) -  \frac {h^2} {L^2}\bigg)  \frac 1 {h^2} \int_{S_i} f\,dy,
\]
where the periodic domain $\Om$ has been split into $M^2$ identical squares $S_i$ with disjoint boundaries and side lengths $h$. The second term in the difference ensures this operator is mean zero. 
It is is bounded in $L^2$  and  satisfies a Poincar\'e-type inequality \cite{AOT}
\EQ{ 
\| u - I_h(u)\|_{L^2} \leq C h \|\nb u\|_{L^2}.
}
These considerations are independent of periodicity and apply in any square domain. Project onto low Fourier modes is another example of a type 1 interpolant. However, it requires global knowledge of a flow and therefore is inappropriate for out application.

\subsection{The mobile framework}

We now adapt the operator $I_h$ to our mobile setting. 
Fix $N\in \N\setminus \{1\}$. We split $\Om$ into a grid of $N^2$ squares where $N$ is chosen to equal $2^{\td N}$ for some $\td N$---in other words the side length of our partition is taken from a dyadic scale. Each square will correspond to an observability region. We label these sub-regions $\Om_i$, {$i=1,\ldots, N^2$}. Note that $|\Om_i|= \frac {L^2} {N^2}$. Let $\ell = \sqrt{ |\Om_i|} = L/N$. Let $x^i= (x_1^i,x_2^i)$ be the center of $\Om_i$. Each $\Om_i$ represents an ``observability region,'' i.e.~the domain from which data will be collected at a given time.
Note that this partition does \textit{not} correspond to the fine grid on which we will make observations, namely that corresponding to the length scale $h$ in the definition of $I_h$ above.

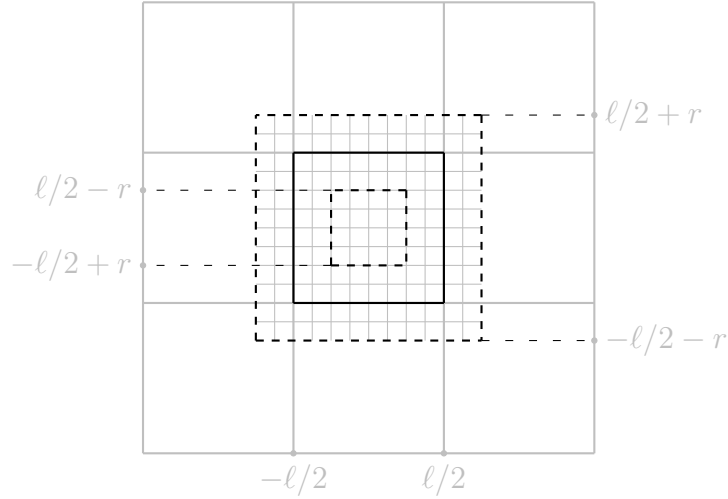
\begin{figure}
\begin{center}
\begin{tikzpicture}
\foreach \i in {0,...,3}
{
        \draw[lightgray, thick] (-3,2*\i-3) -- (3,2*\i-3);
        \draw[lightgray, thick] (2*\i-3,-3) -- (2*\i-3,3);
}
\draw[lightgray] (-1.5, -1.25) -- (1.5, -1.25); 
\draw[lightgray] (-1.5, -1) -- (1.5, -1); 
\draw[lightgray] (-1.5, -.75) -- (1.5, -.75); 
\draw[lightgray] (-1.5, -.5) -- (1.5, -.5); 
\draw[lightgray] (-1.5, -.25) -- (1.5, -.25); 
\draw[lightgray] (-1.5, 0) -- (1.5, 0); 
\draw[lightgray] (-1.5, .25) -- (1.5, .25); 
\draw[lightgray] (-1.5, .5) -- (1.5, .5);
\draw[lightgray] (-1.5, .75) -- (1.5, .75); 
\draw[lightgray] (-1.5, 1) -- (1.5, 1); 
\draw[lightgray] (-1.5, 1.25) -- (1.5, 1.25); 
\draw[lightgray] (-1.25,-1.5) -- (-1.25,1.5);
\draw[lightgray] (-1,-1.5) -- (-1,1.5);
\draw[lightgray] (-.75,-1.5) -- (-.75,1.5);
\draw[lightgray] (-.5,-1.5) -- (-.5,1.5);
\draw[lightgray] (-.25,-1.5) -- (-.25,1.5);
\draw[lightgray] (0,-1.5) -- (0,1.5);
\draw[lightgray] (.25,-1.5) -- (.25,1.5);
\draw[lightgray] (.5,-1.5) -- (.5,1.5);
\draw[lightgray] (.75,-1.5) -- (.75,1.5);
\draw[lightgray] (1,-1.5) -- (1,1.5);
\draw[lightgray] (1.25,-1.5) -- (1.25,1.5);
\draw[black, thick] (-1,-1) -- (-1,1);
\draw[black, thick] (1,1) -- (1,-1);
\draw[black, thick] (1,1) -- (-1,1);
\draw[black, thick] (-1,-1) -- (1,-1);
\draw[dashed, thick] (-1.5,-1.5) -- (-1.5,1.5);
\draw[dashed, thick] (1.5,1.5) -- (1.5,-1.5);
\draw[dashed, thick] (1.5,1.5) -- (-1.5,1.5);
\draw[dashed, thick] (-1.5,-1.5) -- (1.5,-1.5);
\draw[dashed, thick] (-.5,-.5) -- (-.5,.5);
\draw[dashed, thick] (.5,.5) -- (.5,-.5);
\draw[dashed, thick] (.5,.5) -- (-.5,.5);
\draw[dashed, thick] (-.5,-.5) -- (.5,-.5);
\draw[loosely dashed] (1.5,-1.5) -- (3,-1.5);
\draw[loosely dashed] (1.5,1.5) -- (3,1.5);
\draw[loosely dashed] (-.5,-.5) -- (-3,-.5);
\draw[loosely dashed] (-.5,.5) -- (-3,.5);
\filldraw[lightgray] (-1,-3) circle (1pt) node[anchor=north]{$-\ell/2$};
\filldraw[lightgray] (1,-3) circle (1pt) node[anchor=north]{$\ell/2$};
\filldraw[lightgray] (3,-1.5) circle (1pt) node[anchor=west]{$-\ell/2-r$};
\filldraw[lightgray] (3,1.5) circle (1pt) node[anchor=west]{$\ell/2+r$};
\filldraw[lightgray] (-3,.5) circle (1pt) node[anchor=east]{$\ell/2-r$};
\filldraw[lightgray] (-3,-.5) circle (1pt) node[anchor=east]{$-\ell/2+r$};
\end{tikzpicture}
\caption{An illustration of $\Om_i$, which is represented by the solid black square, and adjacent cells. $\Om_i$ is taken to be centered at the origin.   The larger dashed square represents $\td \Om_i$.  The cut-off function
$\phi_i$ is  identically $1$ inside the smaller dashed square and is supported inside $\td \Om_i$. When $\iota (t) = i$, the interpolant $I_{h,t}$ is using observations taken on the fine-scale gray grid; note that observations spill over into a small region outside of $\Om_i$. The lengths of all squares in view are taken along a dyadic scale.
}
\end{center}
\end{figure}

We work with a partition of unity $\{ \phi_i\}_{i=1}^{N^2}$ of $\Om$ so that each $\phi_i$ localizes to $\Om_i$.  The functions $\phi_i$ are chosen to satisfy the following properties (which are illustrated in Figure 1 below):   
\begin{itemize}
    \item Each $\phi_i$ is generated by translating $\phi_1$.
    \item There exists $r < \ell/2$ so that $\phi_1 \equiv 1$ on $[x_1^1-\ell/2+r,x_1^1+\ell/2-r]\times [x_2^1-\ell/2+r,x_2^1+\ell/2-r]$.
    \item $\supp \phi_1 \subset \td \Om_i := [x_1^1-\ell/2-r,x_1^1+\ell/2+r]\times [x_2^1-\ell/2-r,x_2^1+\ell/2+r]$.
    \item $0\leq \phi_1 \leq 1$.
    \item $\sum_i \phi_i  =1$.
    \item We have
    \[
    |\nb \phi_1|\lesssim_{\phi_1} r^{-1} \text{ and }  |\Delta \phi_1|\lesssim_{\phi_1} r^{-2},
    \]
    where the suppressed constants only depend on the function $\phi_1$.
\end{itemize}
 Note the definition of $\td \Om_i$ above as this plays an important role. We fix  $r=\ell/4$ to ensure the boundaries of $\td \Om_i$ line up with dyadic partitions at finer scales.

{We quickly  construct a partition of unity of this form. 
Let $\eta$ be a smooth, radial function with compact support and $\int \eta \,dx= 1$. Let $\eta_\al(x) = \al^{-2} \eta(x/\al)$, which is an approximate identity. Let $\bar \chi_{i}$ be the characteristic function for $\Om_i$. Then $\sum_{i=1}^{N^2} \bar \chi_i = 1$ except on a set of measure zero. Let $\phi_i = \eta_\al * \Om_i$. 
Then
\[
 \sum_{i=1}^{N^2} \phi_i  (x)
 = \int \eta_{\al}(x-y)  \sum_{i=1}^{N^2} \bar \chi_{i}(y)\,dy =1.
\]
The support conditions are satisfied by taking $\al$ small enough.}

We define  a function $\iota(t)$  to be a $\tau_C$-periodic function on $[0,\tau_C)$  which is defined by $\iota (t) = i$ if $t\in [t\in (i-1)/N^2  \tau_C, i/N^2 \tau_C)$, $i=1,\ldots,N^2$.

We  partition each $\Om_i$ into identical squares   of side-length $h$ chosen so that $2^m h=\ell/4$ for some $m\in \N$. Taking the union of these partitions gives a partition of $\Om$ into a grid with sidelength $h$. Note the dyadic  relationship between scales ensures that $\td\Om_i$ is also a union of $h$-length squares which we label $S_j$. For convenience let
\[
\chi_i := \chi_{\td \Om_i}
\]

We define a time-dependent, local interpolant operator based on volume elements as follows: 
\[
 I_{h,t} f :=   \chi_{\td \Om_{\iota(t)}}(x) \sum_{j} \bigg(\chi_{S_j}(x) -  \frac {h^2} {\sqrt {|\td \Om_{\iota (t)}|}}\bigg)  \frac 1 {h^2} \int_{S_j} f(y)\,dy .
\]
Note that $I_{h,t}$ is supported on $\td \Om_{\iota(t)}$. 
Additionally,
\EQ{\label{ineq.L2bounded}
\| I_{h,t} f\|_{L^2} \leq C_I \| f \chi_{\td \Om_{\iota(t)}}\|_{L^2},
}
and  
\EQ{\label{Poincare}
\| u \chi_{\td \Om_{\iota(t)}} - I_{h,t}(u)\|_{L^2} \leq C_I h \|(\nb u)\chi_{\td \Om_{\iota(t)}} \|_{L^2}.
}
The proofs of these estimates are identical to those for the interpolant $I_h$.
Note that the optimal constants in the above estimates are not the same but it is convenient to lump them together under the label $C_I$.

Note that the observation window spends $\tau_C/N^2$ units of time in each of the regions $\Om_i$. Movement from one region to another is discontinuous. It should not be hard to adjust the analytic results in this paper to the case of continuous movement, although the conditions for data assimilation would change. The order that $\Om_i$ are cycled through does not matter.

 For each $\phi_i$ we denote a complimentary cut-off function by
\EQ{\label{defpsi}
\psi_i = \sum_{i\neq j} \phi_j,
}
so that $\psi_i+\phi_i\equiv 1$ on $\Om$.
We say $\Om_i$ is \textit{dominant} if 
\[
\int |w|^2 \phi_i \,dx\geq \frac 1 {N^2} \int |w|^2\,dx
\]
and \textit{active} if 
\[
\int |w|^2 \phi_i \,dx\geq \frac 1 {c_0 N^2} \int |w|^2\,dx,
\]
for some $c_0>1$. This parameter appears in the statement of Theorem \ref{theorem} below---larger values of $c_0$ make the conditions in the theorem more stringent.
Clearly, sine $\chi_i \geq \phi_i$, if $\Om_i$ is active, then we also have  
\[
\int |w|^2 \chi_i \,dx \geq \frac 1 {c_0N^2} \int |w|^2\,dx.
\] 
Alternatively we say $\Om_i^c$ is \textit{codominant} if
\[
\int |w|^2 \psi_i \,dx< \bigg( 1- \frac 1 {N^2}\bigg) \int |w|^2\,dx
\]
and \textit{coactive} if 
\[
\int |w|^2 \psi_i \,dx< \bigg( 1- \frac 1 {c_0N^2} \bigg) \int |w|^2\,dx.
\]
Plainly,  $\Om_i$ is dominant if and only if $\Om_i^c$ is codominant and  $\Om_i$ is active  if and only if $\Om_i^c$ is coactive.

\subsection{Data assimilation}
 
Suppose that $u$ solves the 2D Navier-Stokes equations, written in projected form,
\EQ{\label{eq:NS}
u_t -\nu \Delta u +\mathbb P (u\cdot\nb u) = f; \qquad \nb \cdot u=0,
}
where $\mathbb P$ is the Leray projection operator. It is classical---see, e.g., \cite{cf}---that for large enough times any solution $u$ becomes bounded solely by quantities determined by $\nu$ and $f$. In particular, the estimates  \eqref{inclusions2} stated below hold. Throughout this paper we take this to be true starting from time $t=0$.
Note that since we are on a periodic domain $\Delta \mathbb P = \mathbb P \Delta$.
The nudged equation for $u$ is  
\EQ{ \label{eq:DA}
v_t -\nu \Delta v +\mathbb P  (v\cdot\nb v)= f  -\mu \mathbb P I_{h,t}(v-u); \qquad \nb \cdot v=0.
}
For reasonable choices of $u_0$ and $f$, and certain conditions on $h$ and $\mu$, this system admits a unique global smooth solution---see Remark \ref{remark:Existence}.
The difference $w=u-v$ satisfies
\EQ{\label{eq:difference}
w_t -\nu \Delta w -\mathbb P(w\cdot\nb w - w\cdot \nb u-u\cdot\nb w )=  - \mu \mathbb P I_{h,t}(w ); \qquad \nb \cdot v=0.
}

In the above we will take $f$ to be sufficiently regular. Any solution $u$ is eventually controlled by the Grashof number. We assume that this control holds at all positive times, which can be achieved by initiating our problem at a sufficiently large time for $u$.

\begin{theorem}\label{theorem}
  Let $f\in L^\I(0,\I;L^2)$ be divergence free and mean zero. Let $u$ be a solution to  \eqref{eq:NS} with forcing $f$. Let the Grashoff number be defined by
  \[
  G:= \limsup_{t\to \I}\frac {2} {\nu^2\la_1}  \|f(t)\|_{L^2}.
  \]
  Under these assumptions, \eqref{eq:DA} is well-posed with a unique strong solution.
  Let $v$ be the solution to \eqref{eq:DA} for $u$ and $f$ with initial data identically zero. Let $c_*>0$ and $c_0>1$ be given. 
  If
  \begin{align*}
 \mu  &\geq \max \bigg\{\nu \la_1 G^2 ,c_0 N^2 ( 4\la_1 C_L^4 \nu G^2 +c_* )\bigg\},
\\ h &\leq \min \bigg\{\frac {\nu \la_1^{1/2}} {4C_I\mu} , \sqrt{\frac {\nu} {2C_I^2\mu}} \bigg\},
\\ \tau_C &\lesssim \frac 1 {\mu} e^{- 2C_L^4(G^2+2)},
  \end{align*}
where the suppressed constant only depends  on $c_*$, $c_0$, $\la_1$, $\nu$, $C_L$, $C_I$, $N$, $L$ and $\phi_1$,
then,   
\EQ{
\| (u-v)(t)\|_{L^2}^2 \leq   e^{C_L^4}  \nu^2 G^2 e^{-c_* t/2}.
}    
\end{theorem}

 The suppressed constant in the constraint on $\tau_C$ can be extracted from \eqref{C5}, \eqref{C6} and \eqref{C7} below.

 \begin{remark}
As is well-known, data assimilation results resembling \cite{AOT} imply the  classical determining parameter theorems of \cite{FT,FTiti,JonesTiti,JonesTiti2}. These state that if two solutions to \eqref{eq:NS} agree at a course resolution, then they converge to each other exponentially. In our case, Theorem \ref{theorem} implies a local version of such results. More precisely, if $u$ and $\td u$ are two solutions to \eqref{eq:NS} with the same forcing satisfying the assumptions for $u$ and $f$ in Theorem \ref{theorem}, then, provided
\[
I_{h,t}(u-\td u) (t)=0,
\]
for all $t>0$, we  have that
$\| (u - \td u) (t)\|_{L^2} \to 0$ exponentially as $t\to \I$.
The precise statement of such a corollary can be improved---e.g., by assuming $I_{h,t}(u-\td u) (t) \to 0$ as $t\to \I$---but this would require additional work which is redundant to the existing literature.
\end{remark}

\subsection{Organization and discussion of the proof} Global \textit{a priori} estimates are worked out in Section \ref{sec.global}. In Section \ref{sec.local}, local estimates are obtained which show that, if a region is dominant at time $t_1$, then it remains active for a short period of time assuming the Dirichlet quotient is suitably bounded. In Section \ref{sec.da}, we prove Theorem \ref{theorem}. The proof involves two cases. The first case holds when the Dirichlet quotient is suitably bounded, and so, due to the work in Section \ref{sec.local}, at least one region is active for a short period of time. This becomes our cycling time. If the interpolant  cycles through all regions fast enough, then it must be in an active region for a short time. For these times, nudging is strong enough to drive synchronization across the full cycle time. In the second scenario, the Dirichlet quotient is bounded below which ensure dissipation is strong enough to drive synchronization without using the data assimilation term.

\section{Global \textit{a priori} estimates}\label{sec.global}

In this section we establish \textit{a priori} bounds for solutions to \eqref{eq:DA} and \eqref{eq:difference}.  For a solution $u$ to \eqref{eq:NS}, it is well known that, if $u_0$ is divergence free, mean zero and belongs to $H^1$, and $f$ is divergence free, mean zero and belongs to $L^\I(0,\I;L^2)$, then there exists a unique solution $u$ to \eqref{eq:NS} which has zero mean. This solution satisfies a number of properties.   We make use of the following which are taken directly from \cite{AOT}---proofs can be found in \cite{cf,Temam,Robinson} among other references.  
For any time $T$ we have 
    \EQ{\label{inclusions1}
    u\in C([0,T];H^1) \cap L^2(0,T;H^2);\quad \frac {du} {dt} \in L^2(0,T;L^2).
    }
Additionally, 
\EQ{\label{inclusions2}
\|u\|_{L^2}^2 \leq  \nu^2 G^2;\qquad \|\nb u\|_{L^2}^2 \leq   \nu^2 \la_1 G^2;\qquad \int_t^{t+T}\| \Delta u\|_{L^2}^2\,ds \leq (1+T \nu \la_1)\nu \la_1 G^2.
} 
We take these properties to hold from time $t=0$.

Following the classical literature, \textit{a priori} estimates are obtained from the Gr\"onwall inequality (see, e.g., \cite{Robinson}) which states that
\[
x(t)\leq x(0)e^{G(t)} +\int_0^t e^{G(t)-G(s)}h(s)\,ds; \qquad 
G(t)=\int_0^t g(r)\,dr,
\]
provided  
\[
\frac {dx} {dt}\leq g(t)x +h(t).
\]
If $g=a$ and $h=b$ are constant, then we have
\[
x(t)\leq \bigg(x_0+\frac b a\bigg)e^{at}- \frac b a. 
\]

\subsection{Energy estimate for $v$}

For the solution $v$ to \eqref{eq:DA}, we have \EQ{\label{eq:8.9.22}
\frac 12\frac d {dt} \|v\|_{L^2}^2 +\nu \|\nb v\|_{L^2}^2 &= \langle f+\mu I_{h,t} (u),v\rangle - \mu \langle I_{h,t} (v) , v\rangle.
}  
Using Young's inequalities and \eqref{ineq.L2bounded}, we have  
\EQ{\label{ineq.8.9.22}
\langle f+\mu I_{h,t} (u),v\rangle 
&\leq \frac 1 {\nu \la_1} \|f+\mu I_{h,t}(u)\|_{L^2}^2 + \frac {\nu \la_1} 4 \|v\|_{L^2}^2
\\&\leq \frac {2{C_I}^2} {\nu \la_1} (\|f\|_{L^2}^2 + \mu^2 \|u\|_{L^2}^2) + \frac {\nu \la_1} 4 \|v\|_{L^2}^2.
}
On the other hand
\[
\mu \langle I_{h,t} (v) , v\rangle = \mu \langle I_{h,t} (v) -v \chi_{\iota(t)}, v\rangle + \mu\int |v|^2\chi_{\iota(t)}\,dx.
\]
The   second term above has a good sign. By \eqref{Poincare}, the remaining part satisfies 
\[
\mu \langle I_{h,t} (v) -v\chi_{\iota(t)}, v\rangle \leq C_I \mu h\|\nb v\|_{L^2} \|v\|_{L^2} \leq C_I \mu h \la_1^{-1/2} \|\nb v\|_{L^2}^2.
\]
Provided, 
\begin{equation}\label{C1}\tag{C1}{C_I \mu h \la_1^{-1/2} \leq \frac \nu 4,}
\end{equation}
it follows that 
\EQ{\label{ineq:v.differential}
\frac d {dt} \|v\|_{L^2}^2  +\frac {\nu \la_1} 2 \|v\|_{L^2}^2 + \frac \nu 8 \|\nb v\|_{L^2}^2\leq  \frac {4C_I^2}{\nu\la_1}(\|f\|_{L^2}^2 +\mu^2 \|u\|_{L^2}^2),
}
and, by Gr\"onwall, we conclude
\EQN{
\| v(t)\|_{L^2}^2 \leq e^{-\nu \la_1t}\|v(0)\|_{L^2}^2 + \frac {4C_I^2} {\nu \la_1} (\|f\|_{L^2}^2 +\mu^2 \|u\|_{L^2}^2) (1-e^{-\nu \la_1 t}).
}
If $v(0)=0$ then, using the uniform bound on $\|u\|_{L^2}$ in \eqref{inclusions2}, we get
\[
\| v(t)\|_{L^2}^2 \leq 4C_I^2 [\nu^2+(\mu/\la_1)^2]G^2.
\]
 Integrating \eqref{ineq:v.differential} we also obtain the bound
 \[
 \frac \nu 8 \int_0^T \|\nb v\|_{L^2}^2\,ds \leq \|v_0\|_{L^2}^2 +   \frac {4C_I^2T}{\nu\la_1}(\|f\|_{L^2}^2 +\mu^2 \|u\|_{L^2}^2).
 \]
\begin{remark}\label{remark:Existence}
Note that the above estimates are strong enough to prove  estimates, existence and uniqueness as in  \cite[Theorem 5]{AOT}. For mean zero data and forcing, $v$ will be mean zero. Since the proofs are identical to the existing literature, we omit the details.  
\end{remark} 

\subsection{Energy estimate for $w$}
For the solution $w=u-v$ to \eqref{eq:difference}, 
we have by Ladyzhenskaya's inequality \eqref{ineq.lady} that
\[
\int (w\cdot \nb u)\cdot w\,dx \leq C_L^2\|w\|_{L^2}\|\nb w\|_{L^2} \|\nb u \|_{L^2}.
\]
Hence,
\EQN{
\frac 1 2\frac d {dt} \|w\|_{L^2}^2 +\nu \|\nb w\|_{L^2}^2 &= -\int (w\cdot \nb u)\cdot w\,dx - \mu (I_{h,t} (w  )-w\chi_{\iota(t)}, w ) - \mu \int |w|^2\chi_{\iota(t)}\,dx 
\\&\leq \frac {C_L^4} {2\nu} \|\nb u\|_{L^2}^2 \|w\|_{L^2}^2 + \frac \nu 2 \|\nb w\|_{L^2}^2 +C_I^2 {\mu h^2} \|(\nb w)\chi_{\iota(t)}\|_{L^2}^2
\\&\quad + \mu \int |w|^2\chi_{\iota(t)}\,dx  - \mu \int |w|^2\chi_{\iota(t)}\,dx,
}  
{where we used the property \eqref{Poincare}.}
Choosing 
\begin{equation}\tag{C2}\label{C2}
 {C_I^2 {\mu h^2}\leq \frac \nu 2,}
\end{equation}
and using \eqref{inclusions2}, we have 
\EQN{
\frac d {dt} \|w\|_{L^2}^2     &\leq   \frac {C_L^4} \nu \|\nb u\|_{L^2}^2 \|w\|_{L^2}^2
\leq C_L^4 \nu \la_1 G^2 \|w\|_{L^2}^2
}
It follows that for $t\geq t_0$ 
\EQ{\label{ineq:EnergyBoundw}
\|w(t)\|_{L^2}^2 \leq \|w(t_0)\|_{L^2}^2e^{C_L^4 \nu \la_1 G^2(t-t_0) }.
}

\subsection{Enstrophy estimate}

For the enstrophy, by a standard cancellation in the nonlinearity in the periodic setting, see \cite[(14)-(15)]{AOT} as well as \cite[p.~294]{AOT}, 
\EQN{
\frac 1 2\frac d {dt} \|\nb w\|_{L^2}^2  +\nu\|\Delta w\|_{L^2}^2 &= -\int (w\cdot \nb w )\cdot\Delta u\,dx - \mu (I_{h,t} (w  ) , \Delta w ).
}
We have by \eqref{ineq.lady} and Poincar\'e's inequality that
\EQN{
\bigg|\int (w\cdot \nb w )\cdot\Delta u\,dx\bigg|&\leq \|w\|_{L^4} \| \nb w\|_{L^4} \|\Delta u\|_{L^2}
\\&\leq C_L^2 \|w\|_{L^2}^{1/2} \|\nb w\|_{L^2} \|\Delta w\|_{L^2}^{1/2} \|\Delta u\|_{L^2}
\\&\leq C_L^2\la_1^{-1/2}   \|\nb w\|_{L^2}\|\Delta w\|_{L^2}  \|\Delta u\|_{L^2}
\\&\leq \frac \nu 2 \| \Delta w\|_{L^2}^2 + \frac {C_L^4}{2\nu \la_1} \|\nb w\|_{L^2}^2 \| \Delta u\|_{L^2}^2.
}
Hence,
\EQN{
\frac 1 2 \frac d {dt} \|\nb w\|_{L^2}^2  +\nu\|\Delta w\|_{L^2}^2 &\leq \frac \nu 2 \| \Delta w\|_{L^2}^2 + \frac {C_L^4}{2\nu \la_1} \|\nb w\|_{L^2}^2 \| \Delta u\|_{L^2}^2
 - \mu (I_{h,t} (w ) , \Delta w )
 \\&\leq   \nu  \| \Delta w\|_{L^2}^2 + \frac {C_L^4}{2\nu \la_1} \|\nb w\|_{L^2}^2 \| \Delta u\|_{L^2}^2 + \frac {\mu^2 {C_I}^2} {2\nu} \|w\|_{L^2}^2,
} 
where we used \eqref{ineq.L2bounded} and Young's inequality.
We then have by Gr\"onwall that, for $t>t_0$, 
\EQN{
  \|\nb w(t)\|_{L^2}^2    &\leq   \| \nb w(t_0)\|_{L^2}^2 e^{\frac {C_L^4} {\nu \la_1}\int_{t_0}^t\|\Delta u(s)\|_{L^2}^2\,ds }
  \\&+ \frac {\mu^2C_I^2} \nu \int_{t_0}^t e^{ C_L^4 (\nu \la_1)^{-1} \int_{s}^t \|\Delta u(r)\|_{L^2}^2 \,dr  }\|w(s)\|_{L^2}^2\,ds
 \\&\leq 
   \| \nb w(t_0)\|_{L^2}^2 e^{C_L^4G^2(1+\nu \la_1 (t-t_0))}
 \\&+ \frac {\mu^2C_I^2} \nu (t-t_0) \sup_{t_0<s<t}\|w(s)\|_{L^2}^2 e^{C_L^4G^2(1+\nu\la_1 (t-t_0))},
}
where we used \eqref{inclusions2}.
Using \eqref{ineq:EnergyBoundw} and assuming that 
\begin{equation}
    {t-t_0\leq  \mu^{-1},} \tag{C3}\label{C3}
\end{equation} we obtain
\EQ{\label{ineq:EnstrophyBoundw.0}
  \|\nb w\|_{L^2}^2(t)    &\leq  
\| \nb w(t_0)\|_{L^2}^2 e^{C_L^4 G^2(1+\nu\la_1 \mu^{-1}) }
\\& +\frac {\mu C_I^2} \nu  \|w(t_0)\|_{L^2}^2 e^{C_L^4 \nu \la_1 G^2\mu^{-1} }   e^{C_L^4G^2(1+\nu\la_1 \mu^{-1})}.
}
If \begin{equation}\label{C4}\tag{C4}{\nu \la_1 G^2 \leq \mu},\end{equation} then this simplifies to 
\EQ{\label{ineq:EnstrophyBoundw}
  \|\nb w\|_{L^2}^2(t)    &\leq
\| \nb w(t_0)\|_{L^2}^2 e^{C_L^4(G^2   +1) }
+\frac {\mu C_I^2}\nu \|w(t_0)\|_{L^2}^2  e^{C_L^4(G^2+2)}.
}

\subsection{Lower bound for $\frac d {dt} \|w\|_{L^2}^2$}
We also require a lower bound for $\frac d {dt} \|w\|_{L^2}^2$. To obtain this note that
\EQN{
\frac 1 2 \frac d {dt} \|w\|_{L^2}^2 &\geq - \nu \|\nb w\|_{L^2}^2 -\bigg|\int (w\cdot \nb u)\cdot w\,dx\bigg|  - \big| \mu \langle I_{h,t} (w  ) , w \rangle \big| 
\\&\geq -\nu \|\nb w\|_{L^2}^2 -C_L^2\|w\|_{L^2}\|\nb w\|_{L^2}\|\nb u\|_{L^2} -\mu C_I \|w\|_{L^2}^2
\\&\geq -\nu \|\nb w\|_{L^2}^2 - \frac {C_L^4}{\nu} \|w\|_{L^2}^2 \|\nb u\|_{L^2}^2 - \nu \|\nb w\|_{L^2}^2 -\mu C_I \|w\|_{L^2}^2
\\&\geq -2 \nu \|\nb w\|_{L^2}^2 -\frac {C_L^4}{\nu} \|w\|_{L^2}^2 \|\nb u\|_{L^2}^2 -\mu C_I \|w\|_{L^2}^2.
}
Then, by \eqref{C4},
\EQ{\label{lower}
\frac d {dt} \|w\|_{L^2}^2 &\geq -4 \nu \|\nb w\|_{L^2}^2 -2C_L^4 \nu\la_1 G^2 \|w\|_{L^2}^2 - 2\mu C_I\|w\|_{L^2}^2
\\&\geq -4\nu \|\nb w\|_{L^2}^2 - \mu (2C_L^4 +2C_I) \|w\|_{L^2}^2.
}

\subsection{Bound for the Dirichlet quotient}
The Dirichlet quotient $\mathcal Q$ is defined to be 
 \EQN{
 \mathcal Q(t):=\frac{\|\nb w\|_{L^2}^2} {\|w\|_{L^2}^2}.
 }
We introduce two scenarios:
\begin{itemize}
    \item Scenario 1: all $t$ such that  $\mathcal Q(t) \leq \mu /\nu  $ and,
    \item Scenario 2: all $t$ such that $\mathcal Q(t) > \mu /\nu$.
\end{itemize}

\begin{lemma} \label{lemma.Dirichlet}Let $\mu$ satisfy \eqref{C4} and $h$ satisfy \eqref{C2}. 
Assume that $t_1$ is a scenario 1 time.
Let 
\EQN{
\tau_{\mathcal Q} = \frac 1 {2\mu}   \big( (8C_I^2 +2C_L^4 +2C_I)e^{C_L^4 (G^2+1)} \big)^{-1}    .
}
Then, for all $t\in [t_1,t_1+\tau_{\mathcal Q}]$,  
 \EQN{
 \mathcal Q(t) 
 \leq 4C_I^2 \frac {\mu } \nu e^{C_L^4(G^2+2)}.
 }
\end{lemma}

\begin{proof}
We have  from \eqref{ineq:EnstrophyBoundw} that, for $t_2-t_1\leq \mu^{-1}$, which is exactly \eqref{C3}, and assuming $t_1$ is a Scenario 1 time,
\EQ{\label{ineq:scenario1enstrophy}
 \|\nb w\|_{L^2}^2(t_2)    &\leq 2 \frac {\mu C_I^2} \nu
\| w(t_1)\|_{L^2}^2 e^{C_L^4(G^2   +2) }.
}
From the mean value theorem, \eqref{lower}, \eqref{C3} and \eqref{C4},   we have
\EQ{\label{ineq:lowerboundondifferencequotient}
\|w(t_2)\|_{L^2}^2 - \|w(t_1)\|_{L^2}^2 &\geq (t_2-t_1)\inf_{t_1<t<t_2} \frac d {dt} \|w\|_{L^2}^2(t) 
\\& \geq (t_2-t_1) \inf_{t_1<t<t_2}(-4\nu \|\nb w\|_{L^2}^2 -(2C_L^4 +2C_I)\mu \|w\|_{L^2}^2).
\\&\geq - (t_2-t_1) \mu \| w(t_1)\|_{L^2}^2 \big( 8C_I^2 e^{C_L^4 (G^2+1)}+(2C_L^4 +2C_I)e^{C_L^4 \nu\la_1 G^2 (t_2-t_1)} \big)
\\&\geq - (t_2-t_1) \mu \| w(t_1)\|_{L^2}^2 \big( (8C_I^2 +2C_L^4 +2C_I)e^{C_L^4 (G^2+1)} \big),
}
where we used   \eqref{ineq:scenario1enstrophy} and \eqref{ineq:EnergyBoundw}.
 It follows that
\EQ{\label{ineq.aaa}
\|w(t_2)\|_{L^2}^2  \geq \bigg(1-   (t_2-t_1) \mu  
 \big( (8C_I^2 +2C_L^4 +2C_I)e^{C_L^4 (G^2+1)} \big)\bigg) \|w(t_1)\|_{L^2}^2.
}
If
 \begin{equation} 
 \label{D5}
   t_2-t_1 \leq \frac 1 {2\mu}     \big( (8C_I^2 +2C_L^4 +2C_I)e^{C_L^4 (G^2+1)} \big)^{-1},   
 \end{equation}
 then 
 \EQ{\label{ineq:scenario1energy}
 \|w(t_2)\|_{L^2}^2 \geq \frac 1 2 \|w(t_1)\|_{L^2}^2.
}
Note that \eqref{D5} implies \eqref{C3}.
 So,   using \eqref{ineq:scenario1enstrophy} in the numerator and \eqref{ineq:scenario1energy} in the denominator, the Dirichlet quotient $\mathcal Q(t_2)$ is bounded by 
 \EQN{
 \mathcal Q(t_2)  \leq 4 \frac {\mu C_I^2} \nu e^{C_L^4(G^2+2)}
,
 }
 provided $t_2-t_1\leq \tau_C$.

\end{proof}



\section{Local estimates and active regions}\label{sec.local}

We make use of the following \textit{local energy equality},
\EQ{\label{eq:le}
&\frac 1 2\frac d {dt} \| w\psi^{1/2} \|_{L^2}^2 =   \langle   {\nu} \Delta w +  \mathbb P (w\cdot \nb w -u\cdot\nb w - w\cdot\nb u) - \mu \langle \mathbb P I_{h,t} (w)          , w\psi \rangle
}
where $\psi$ is a smooth non-negative function with compact support. This is obtained from \eqref{eq:DA} by  testing against $w\psi$. In this section we establish an upper bound for the left-hand side of \eqref{eq:le}. By taking $\psi = \psi_i$ as defined in \eqref{defpsi},   this will imply that if a region is dominant at a given time $t_1$, then it remains active at least for a period of time depending on the parameters of the problem and $\|w(t_1)\|_{L^2}$.

\begin{lemma}\label{lemma.3.1} Assuming $u,w\in H^1$ and $\psi\in \{\psi_i\}_{i=1}^n$ where $\psi_i$ are defined in \eqref{defpsi}, we have 
\EQN{
\langle \mathbb P (w\cdot \nb w -u\cdot\nb w -w\cdot\nb u), w\psi \rangle &\leq C_L^2 ( \|w\|_{L^2} +4\nu G  ) \|\nb w\|_{L^2}^2.
}
\end{lemma}

 
\begin{proof}

By the boundedness of $\mathbb P$ in $L^p$ for $1<p<\I$, the term in question satisfies
\EQN{
&\|w\cdot\nb w + u\cdot \nb w +w\cdot \nb u\|_{L^{4/3}} \| w \|_{L^4}   
\\&\leq    \big(   \|\nb w \|_{L^2}  \|w \|_{L^4} +\|\nb w \|_{L^2}  \| u\|_{L^4}+\| \nb u\|_{L^2}  \| w\|_{L^4}\big)    C_L  \| w\|_{L^2}^{1/2} \|\nb w\|_{L^2}^{1/2}
\\&\leq C_L^2 \|w\|_{L^2} \|\nb w\|_{L^2}^2 + C_L^2  \| w\|_{L^2}^{1/2} \|\nb w\|_{L^2}^{3/2} \| u\|_{L^2}^{1/2}\|\nb u\|_{L^2}^{1/2}+C_L^2\| \nb u\|_{L^2}       \| w\|_{L^2}  \|\nb w\|_{L^2} 
\\&\leq C_L^2 \|w\|_{L^2} \|\nb w\|_{L^2}^2 + C_L^2  \la_1^{-1/4} \|\nb w\|_{L^2}^{2} \| u\|_{L^2}^{1/2}\|\nb u\|_{L^2}^{1/2}+C_L^2 \la_1^{-1/2}\| \nb u\|_{L^2}       \|\nb w\|_{L^2}^2
\\&\leq C_L^2 \|w\|_{L^2} \|\nb w\|_{L^2}^2 + C_L^2 2\nu   G  \|\nb w\|_{L^2}^{2} +C_L^2   2 \nu   G       \|\nb w\|_{L^2}^2
\\&\leq C_L^2 ( \|w\|_{L^2} +4\nu G  ) \|\nb w\|_{L^2}^2.
}

\end{proof}

We now obtain an estimate on the local rate of change of the energy.

\begin{lemma}\label{lemma:local.bound.time.d}
Assume $u$ solves \eqref{eq:NS}, $v$ solves \eqref{eq:DA} and  $\psi\in \{\psi_i\}_{i=1}^n$ as defined  in \eqref{defpsi}.
Also assume that $t_1$ is a scenario 1 time. 
Then,    for all times $t$ in $(t_1,t_1+\tau_{\mathcal Q})$, and letting $M>0$ satisfy $\sup_{t_1< t< t_1+\tau_{\mathcal Q}}\|w(t)\|_{L^2}^2 \leq M$, we have
\EQN{
\frac d {dt} \| w\psi^{1/2} \|_{L^2}^2  
\leq  \bigg[ c_\psi\frac \nu {2r^2} +\mu C_I +  C_L^2 M^{1/2}  4 \frac \mu \nu C_I^2 e^{C_L^4(G^2+2)} +16 G \mu C_I^2 e^{C_L^4(G^2+2)}  \bigg] \|w\|_{L^2}^2.
}
\end{lemma}

\begin{proof} 
We seek an upper bound for $\frac d {dt} \| w\psi^{1/2} \|_{L^2}^2$ where $\psi\in \{\psi_i \}_{i=1}^N$ is fixed.  Our starting point is the local energy equality \eqref{eq:le}.
We have, after integrating by parts,
\[
\int \Delta w w \psi \,dx = - \int |\nb w|^2\psi\,dx  +  \frac 1 2 \int w^2 \Delta \psi\,dx.
\]
The first term on the right-hand side has a good sign and is dropped.
Plainly \[\nu \int w^2 \Delta \psi\,dx \leq c_\psi \nu r^{-2} \|w\|_{L^2}^2.\]
We also have  
\[
\mu | \langle \mathbb P I_{h,t}w,w\psi \rangle|\leq \mu C_I \|w\|_{L^2}^2.
\]

From these estimates and Lemma \ref{lemma.3.1} we obtain  
\EQN{ 
\frac d {dt} \| w\psi^{1/2} \|_{L^2}^2  
&\leq   \big(c_\psi\frac \nu {2r^2} +\mu C_I\big) \|w\|_{L^2}^2+  C_L^2 (\|w\|_{L^2} +4\nu G)  \|\nb w\|_{L^2}^2
\\&\leq \bigg[ c_\psi\frac \nu {2r^2}  +\mu C_I +  C_L^2 (M^{1/2} +4\nu G) \mathcal Q \bigg] \|w\|_{L^2}^2
\\&\leq \bigg[ c_\psi\frac \nu {2r^2}+\mu C_I  + 4 C_L^2 M^{1/2}   \frac \mu \nu C_I^2 e^{C_L^4(G^2+2)} +16 G \mu C_I^2 e^{C_L^4(G^2+2)}   \bigg] \|w\|_{L^2}^2,
}
where we used Lemma \ref{lemma.Dirichlet}.

\end{proof}

We are now ready to state our main observation about dominant and active regions, namely that, starting at a scenario 1 time, a dominant region remains active for a short period of time $\tau_C$ which can be quantified.

\begin{lemma} \label{lemma:tauC} Fix $c_0>1$.  Let $\mu\geq 1$ satisfy  \eqref{C4} and $h$ satisfy \eqref{C2} and \eqref{C7} below.
 Assume that $t_1$ is a scenario 1 time.
    Assume also that, for some $M>1$ we have 
    \[
 \|w(t_1)\|_{L^2}^2 \leq M.   
    \]
 Take $\tau_C>0$ to  be the minimum of the upper bounds in \eqref{C5}, \eqref{C6} and\eqref{C7} below (note that $\tau_C \sim \mu^{-1}$).
 If a region $\Om_i$ is dominant at time $t_1$ then $\Om_i$ is active for all times in $[t_1,t_1+\tau_C]$. It follows that 
\[
\int |w(t_2)|^2 \chi_i \,dx \geq \frac 1 {c_0 N^2} \| w\|_{L^2}^2 (t_2).
\]
 \end{lemma}

\begin{remark}\label{remark1}
Above $\tau_C$ depends on $M=M(t_1)$. In principle $\tau_C$ could collapse as time passes since \eqref{ineq:EnergyBoundw} allows for the growth of the energy of $w$.
In our application, we will iteratively show that the energy is decaying exponentially as $t_1$ grows. So, this dependence on the energy will not lead to $\tau_C\to 0$ when we iterate to extend our result to all times.
\end{remark}

\begin{proof}

Assume $t_1$ is a scenario 1 time and without loss of generality take $\Om_i=\Om_1$ to be dominant at time $t_1$.
From {Lemma \ref{lemma.Dirichlet}}, we have that $\mathcal Q(t)\leq 4 C_I^2 \frac \mu \nu e^{C_L^4(G^2+2)}$ on $[t_1,t_1+\tau_{\mathcal Q}]$.   
Then, by the mean value theorem, \eqref{ineq:EnergyBoundw} and Lemma \ref{lemma:local.bound.time.d}, we have
\EQ{\label{def.constantTilde}
&\| w \psi_1^{1/2}\|_{L^2}^2 (t_2)-\| w \psi_1^{1/2}\|_{L^2}^2 (t_1) \leq (t_2-t_1) \sup_{t_1<t<t_2} \frac d {dt} \|w\psi_1^{1/2}\|_{L^2}^2
\\&\leq (t_2-t_1)  \bigg[ c_\psi\frac \nu {2r^2}+\mu C_I  + 4  C_L^2 M^{1/2}   \frac \mu \nu  C_I^2 e^{C_L^4(G^2+2)} +16 G \mu C_I^2 e^{C_L^4(G^2+2)}   \bigg] \sup_{t_1<t<t_2} \|w(t)\|_{L^2}^2
\\&\leq (t_2-t_1)  {\bigg[ c_\psi\frac \nu {2r^2} +\mu C_I + 4 C_L^2 M^{1/2}   \frac \mu \nu C_I^2 e^{C_L^4(G^2+2)} +16 G \mu C_I^2 e^{C_L^4(G^2+2)}  \bigg]e^{C_L^4 \nu \la_1 G^2(t_2-t_1) }}
\\&\quad \cdot \|w(t_1)\|_{L^2}^2
\\&\leq (t_2-t_1) \underbrace{\bigg[ c_\psi\frac \nu {2r^2}+\mu C_I  + 4 C_L^2 M^{1/2}    \frac \mu \nu C_I^2 e^{C_L^4(G^2+2)} +16 G \mu C_I^2 e^{C_L^4(G^2+2)}   \bigg]e^{C_L^4}}_{=:\mu K}
\\&\quad \cdot \|w(t_1)\|_{L^2}^2,
} 
where we assumed   $t_2-t_1<\tau_C$ where 
\begin{equation} \tag{C5}\label{C5}
{\tau_C\leq \tau_{\mathcal Q},}
\end{equation}
as this implies $\tau_C  \leq \mu^{-1}$.
We additionally require
\begin{equation}\tag{C6}\label{C6}
{ \tau_C 
\leq \frac{\ga}{\mu K}\bigg( 1-\frac {1} {N^2}\bigg),}
\end{equation}
where
\EQ{\label{def:gamma}
\ga=\ga(N,c_0):=\frac 1 2 \bigg(\frac {1-(\sqrt{c_0}N)^{-2}} {1- N^{-2}} -1 \bigg)>0.
}
Then, since $\Om_1^c$ is codominant at time $t_1$, we get from \eqref{def.constantTilde} that
\EQN{
\| w \psi_1^{1/2} \|_{L^2}^2 (t_2)
&\leq \| w \psi_1^{1/2}\|_{L^2}^2(t_1)   + \ga \bigg(1-\frac {1} {N^2}\bigg) \|w(t_1)\|_{L^2}^2
\\&\leq (1+\gamma )\bigg(1-\frac 1 {N^2}\bigg)  \|w(t_1)\|_{L^2}^2.
}
Note that $\ga$ is chosen so that the above prefactor
is the midpoint between 
$
1-N^{-2}$ and $1-(\sqrt {c_0}N)^{-2}
$.
We next guarantee
\EQ{\label{ineq:goal}
(1+\gamma )\bigg(1-\frac 1 {N^2}\bigg) \|w(t_1)\|_{L^2}^2 \leq \bigg( 1 -\frac 1 {c_0N^2} \bigg)  \|w(t_2)\|_{L^2}^2,
}
again by controlling $\tau_C$.
If $\| w(t_1)\|_{L^2}^2 \leq \| w(t_2)\|_{L^2}^2$, then we are done by our choice of $\ga$ and $\tau_C$ does not need to be updated.  
Otherwise, noting that   \eqref{ineq.aaa} and $t_2-t_1<\tau_C$ imply
\EQN{
\|w(t_1)\|_{L^2}^2< 
\bigg(1-  \tau_C \mu  
   (8C_I^2 +2C_L^4 +2C_I)e^{C_L^4 (G^2+1)}  \bigg)
\|w(t_2)\|_{L^2}^2,
}
we see that \eqref{ineq:goal} is met provided
\[
{(1+\gamma )\bigg(1-\frac 1 {N^2}\bigg)  \bigg(1-  \tau_C \mu    (8C_I^2 +2C_L^4 +2C_I)e^{C_L^4 (G^2+1)} \bigg)\bigg)^{-1}\leq\bigg(1 - \frac 1{c_0N^2}\bigg)},
\]
which, in turn, holds if
\begin{equation}\tag{C7}\label{C7}
{
\tau_C \leq \frac 1 {\mu    (8C_I^2 +2C_L^4 +2C_I)e^{C_L^4 (G^2+1)} \bigg)}\bigg( 1- \frac {(1+\ga)(1-N^{-2})} {1-(\sqrt{c_0}N)^{-2}}\bigg).
}
\end{equation}
We therefore take $\tau_C$ to be the minimum of the quantities listed above.

We have thus shown
\[
\| w \psi_1^{1/2}(t_2)\|_{L^2}^2  <  \bigg( 1 -\frac 1 {c_0N^2} \bigg) \| w(t_2)\|_{L^2}^2 ,
\]
that is, $\Om_1^c$ is coactive at time $t_2$. 
Using the fact that $\phi_1+\psi_1 =1$, we have
\[
\|w\phi_1^{1/2}(t_2)\|_{L^2}^2\geq \frac 1 {c_0N^2} \| w(t_2)\|_{L^2}^2 ,
\] 
and therefore $\Om_1$ is active at time $t_2$. 

\end{proof}

\section{Mobile data assimilation}\label{sec.da}

\subsection{Short time synchronization}We now show that the data assimilation equation synchronizes regardless of whether we are in scenario 1 or scenario 2.

\begin{lemma}\label{lemma:sync1}
Assume that $t_1$ is a scenario 1 time. Let $c_*>0$ and $c_0>1$ be given. Then, choosing $\mu$ to satisfy  \eqref{C4}, \eqref{C8} and \eqref{C9},
and $h$ to satisfy \eqref{C1} and \eqref{C2},
and letting $\tau_C$ be as in Lemma \ref{lemma:tauC},
it follows that
\[
 \|w(t_1+\tau_C)\|_{L^2}^2 \leq  \| w(t_1) \|_{L^2}^2e^{-c_* \tau_C}.
\]
\end{lemma}

The idea behind this lemma is that, as the observability window cycles through all regions over the interval $[t_1,t_1+\tau_C]$, it remains in an active interval for $\tau_C/N^2$ units of time. When this is the case, synchronization is driven by the nudging term at an exponential rate. Over the remaining times, $\|w\|_{L^2}^2$ might grow, but not enough to overcome the convergence during synchronization.

\begin{proof}
Suppose that $\Om_i$ is the dominant region at time $t_1$. Then, by Lemma \ref{lemma:tauC}, which requires $\mu$ satisfy \eqref{C2} and \eqref{C4} and $h$ satisfies \eqref{C2}, it is active on $[t_1,t_1+\tau_C]$ and, therefore,
\[
\int |w|^2 \chi_i \,dx \geq \frac 1 {c_0 N^2} \int |w|^2\,dx.
\]
There exists a strict and maximal subinterval $I$ of $[t_1,t_1+\tau_C]$ on which the interpolant operator localizes to $\Om_i$. We may decompose $[t_1,t_1+\tau_C]$ into the three intervals $[t_1,t_2]$, $I = [t_2,t_3]$, and $[t_3,t_1+\tau_C]$, where $t_1\leq t_2\leq t_3$. The first or last interval can be degenerate.  Note that $t_3-t_2 = \tau_C N^{-2}$. On $I$, since $\Om_i$ is active,  by the Poincar\'e inequality  we have  
\EQ{ \label{ineq.energy.se}
\frac 1 2\frac d {dt} \|w\|_{L^2}^2 +\nu \|\nb w\|_{L^2}^2
&=-\int (w\cdot \nb u)\cdot w\,dx - \mu (I_{h,t} (w   )-w\chi_{\iota(t)}, w ) - \mu \int |w|^2\chi_{\iota(t)}\,dx
\\&\leq C_L^2\|\nb u\|_{L^2} \| w\|_{L^2} \|\nb w\|_{L^2} + \mu C_Ih \|w\|_{L^2}\|\nb w\|_{L^2} - \frac \mu  {c_0 N^2} \int |w|^2 \,dx
\\&\leq  \frac 1 \nu C_L^4  \|\nb u\|_{L^2}^2  \|w\|_{L^2}^2 +  \frac \nu 2 \|\nb w\|_{L^2}^2 - \frac \mu  {c_0 N^2} \int |w|^2 \,dx,
}
where we are using \eqref{C1}.
Then 
\EQN{ 
\frac d {dt} \|w\|_{L^2}^2  & \leq  2(4C_L^4 \nu \la_1 G^2 - \frac {\mu} {c_0N^2} )\|w\|_{L^2}^2 
\\&\leq \frac {-\mu} {c_0 N^2}\|w\|_{L^2}^2,
}
provided
\begin{equation} \tag{C8}\label{C8}
{4C_L^4 \nu \la_1 G^2 \leq \frac {\mu} {c_0N^2}  . }
\end{equation}
Hence for $t\in [t_2,t_3]$ we have
\[
\|w (t)\|_{L^2}^2 \leq \|w(t_2)\|_{L^2}^2 e^{- \frac \mu  {c_0 N^2} (t-t_2)}.
\]
In particular, 
\[
\| w(t_3)\|_{L^2}^2  \leq \|w(t_2)\|_{L^2}^2 e^{- \frac \mu  {c_0 N^2} \frac {\tau_C} {N^2}}.
\]
On the other hand, using \eqref{ineq:EnergyBoundw}, for $t\in [t_i,t_{i+1}]$ where $i=1,3$ we have 
\[
\|w(t)\|_{L^2}^2 \leq \|w(t_i)\|_{L^2}^2e^{4C_L^4\nu\la_1G^2(t-t_i) }.
\]
Hence, 
\EQN{
\|w(t_1+\tau_C)\|_{L^2}^2  &\leq
 \|w(t_3)\|_{L^2}^2e^{4C_L^4\nu\la_1G^2(t_1+\tau_C-t_3)}
\\&\leq 
 \|w(t_2)\|_{L^2}^2e^{4C_L^4\nu\la_1G^2(t_1+t_C-t_i)-  \frac \mu  {c_0 N^2} \frac {\tau_C} {N^2}}
\\&\leq 
\|w(t_1)\|_{L^2}^2e^{4C_L^4\nu\la_1G^2(t_1+\tau_C-t_3 +t_2-t_1)  - \frac \mu  {c_0 N^2} \frac {\tau_C} {N^2}  }
\\&\leq \|w(t_1)\|_{L^2}^2e^{\tau_C( 4C_L^4\nu\la_1G^2- \frac \mu  {c_0 N^4} \big)} 
}
Provided
\begin{equation}\label{C9}\tag{C9}
{
 4C_L^4\nu\la_1G^2 - \frac \mu  {c_0 N^4} \leq -c_*, 
}
\end{equation}
we have 
\EQN{
\|w(t_1+\tau_C)\|_{L^2}^2   \leq e^{ -c_* {\tau_C}  }   \|w(t_1)\|_{L^2}^2.
}
\end{proof}

\begin{lemma}\label{lemma:sync2}
Suppose that scenario 2 holds on $[t_1,t_2]$. Let $c_*>0$ and $c_0>1$ be given.  Then for all $t\in [t_1,t_2]$,
assuming $\mu$ satisfies \eqref{C9},
and $h$ satisfies \eqref{C1}
we have
\[
\|w(t)\|_{L^2}^2 \leq \|w(t_1)\|_{L^2}^2 e^{-c_*(t-t_1)}.
\]
\end{lemma}

\begin{proof}
From  \eqref{ineq.energy.se} and dropping the term with the good sign, we have
\EQ{\label{ineq.no.nudge}\frac 1 2 \frac d {dt} \|w\|_{L^2}^2 +\nu \|\nb w\|_{L^2}^2  &\leq \frac 1 \nu C_L^4  \|\nb u\|_{L^2}^2  \|w\|_{L^2}^2 +  \frac \nu 2 \|\nb w\|_{L^2}^2.
}
It follows that  
\[
\frac d {dt} \|w\|_{L^2}^2 +\mu   \|w\|_{L^2}^2\leq 4 C_L^4 \nu\la_1 G^2 \|w\|_{L^2}^2.
\]
The desired result follows if
\EQN{
 {4C_L^4 \nu \la_1 G^2- \mu  \leq -c_*,}
}
which is implied by \eqref{C9}.

\end{proof}

Note that in \eqref{ineq.no.nudge} we drop the nudging term. Thus, in this lemma we are using diffusion to drive synchronization.

\subsection{Proof of Theorem \ref{theorem}}

\begin{figure}\label{fig2}
\begin{center}
\begin{tikzpicture}
\draw (-5,0) -- (5,0);
\draw (-5,-.4) -- (-5,.4);
\draw (-2,-.4) -- (-2,.4); 
\draw (1,-.4) -- (1,.4); 

\draw [fill=lightgray] (-5,0) -- (-5,.4) -- (-4,.4) -- (-4,0) -- (-5,0);
\draw [fill=lightgray] (-4,0) -- (-4,.4) -- (-3,.4) -- (-3,0) -- (-4,0);
\draw [fill=lightgray] (-3,0) -- (-3,.4) -- (-2,.4) -- (-2,0) -- (-3,0);
\draw [fill=lightgray] (-2,0) -- (-2,.4) -- (-1,.4) -- (-1,0) -- (-2,0);
\draw [fill=lightgray] (.5,0) -- (.5,.4) -- (1.5,.4) -- (1.5,0) -- (.5,0);
\draw [fill=lightgray] (3.5,0) -- (4.5,0) -- (4.5,.4) -- (3.5,.4) -- (3.5,0);

\draw (-2,.4) -- (.5,.4);
\draw (1.5,.4) -- (3.5,.4);
\draw (4.5,.4) -- (5,.4) -- (5,0);
 
 \draw[pattern=north west lines, pattern color=black] (-1,0) rectangle (.5,.4);
  \draw[pattern=north west lines, pattern color=black] (1.5,0) rectangle (3.5,.4);
   \draw[pattern=north west lines, pattern color=black] (4.5,0) rectangle (5,.4);

\draw [fill=gray] (-5,-.4) -- (-2,-.4) -- (-2,0) -- (-5,0) -- (-5,-.4);
\draw [fill=gray] (.5,-.4) -- (1,-.4) -- (1,0) -- (.5,0) -- (.5,-.4);
\draw [fill=gray] (1.1,-.4) -- (1.3,-.4) -- (1.3,0) -- (1.1,0) -- (1.1,-.4);
\draw [fill=gray] (3.5,0) -- (3.7,0) -- (3.7,-.4) -- (3.5,-.4) -- (3.5,0);

\draw (-2,0) -- (.5,0) -- (.5,-.4) -- (-2,-.4) -- (-2,0);
\draw (1,-.4) -- (1.1,-.4);
\draw (1.3,-.4) -- (3.5,-.4);
\draw (3.7,-.4) -- (5,-.4) -- (5,0);

\filldraw[black] (-5,0) circle (2pt);
\filldraw[black] (-4,0) circle (2pt);
\filldraw[black] (-3,0) circle (2pt);
\filldraw[black] (-2,0) circle (2pt);
\filldraw[black] (-1,0) circle (2pt);
\filldraw[black] (.5,0) circle (2pt);
\filldraw[black] (1.5,0) circle (2pt);
\filldraw[black] (3.5,0) circle (2pt);
\filldraw[black] (4.5,0) circle (2pt);

\draw (-2.5,3) node {Use Lemma \ref{lemma:sync1}};
\draw (-2.5,2.5)--(-4.5,.4);
\draw (-2.5,2.5)--(-3.5,.4);
\draw (-2.5,2.5)--(-2.5,.4);
\draw (-2.5,2.5)--(-1.5,.4);
\draw (-2.5,2.5)--(1,.4);
\draw (-2.5,2.5)--(4,.4);

\draw (2.5,3) node {Use Lemma \ref{lemma:sync2}};
\draw (2.5,2.5)--(-.25,.4);
\draw (2.5,2.5)--(2.5,.4);
\draw (2.5,2.5)--(4.75,.4);

\draw (-2.5,-3) node {Scenario 1 times};
\draw (-2.5,-2.5)--(-3.5,-.4);
\draw (-2.5,-2.5)--(.75,-.4);
\draw (-2.5,-2.5)--(1.2,-.4);
\draw (-2.5,-2.5)--(3.6,-.4);

\draw (2.5,-3) node {Scenario 2 times}; 
\draw (2.5,-2.5)--(1.05,-.4);
\draw (2.5,-2.5)--(-.75,-.4);
\draw (2.5,-2.5)--(2.4,-.4);
\draw (2.5,-2.5)--(4.35,-.4);

\draw (6,-2.6) -- (7,-2.6);
\draw (6.5,-3) node {$\tau_C$};
\filldraw[black] (6,-2.6) circle (2pt);
\filldraw[black] (7,-2.6) circle (2pt);
\draw[->,thick] (-5,0) -- (6.2,0);
\draw (6.5,-.4) node {$t$};

\end{tikzpicture}
\caption{The partition scheme for the time axis in the proof of Theorem \ref{theorem}.
The dark gray lower intervals represent scenario 1 times while  scenario 2 times are unfilled. The gray upper intervals, which have length $\tau_C$ and are initiated at scenario 1 times, represent times across which at least one region stays active---see Lemma \ref{lemma:tauC}. On these intervals, nudging drives synchronization---see Lemma \ref{lemma:sync1}. In the hatched upper intervals, which can have any length, dissipation drives synchronization---see Lemma \ref{lemma:sync2}. 
}
\end{center}
\end{figure}
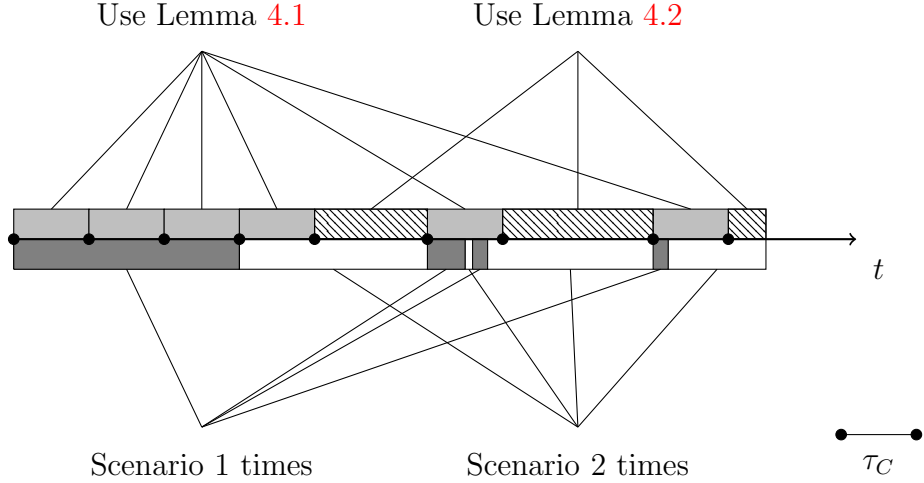

\begin{proof}
Let $w$ be a solution to \eqref{eq:DA} on $\Om\times (0,\I)$ where $v_0=0$ so that $w_0 = u_0$.    Suppose $t=t_0=0$ is a scenario 1 time---if it is not then, by Lemma \ref{lemma:sync2}, there is an interval originating at zero on which $\|w(t)\|_{L^2}^2$ is decreasing exponentially and we re-start this argument at the first scenario 1 time (if no time exists then Lemma \ref{lemma:sync2} applies for all times and we are done).  Let $I=[0,m \tau_C]$ where $m\geq 1$ is the largest integer such that, for all $i=0,\ldots, m-1$ we have $t=i \tau_C$ is a scenario $1$ time. Note that $\|w (i \tau_C)\|_{L^2}$ is a decreasing sequence hence the parameters in Lemma \ref{lemma:sync1}, in particular $\tau_C$, can be chosen uniformly depending on $M=\|u_0\|_{L^2}^2 \leq \nu^2G^2$ (see Remark \ref{remark1}).
Then 
\[
\|w(m \tau_C)\|_{L^2}^2 \leq e^{-c_* m \tau_C} \nu^2G^2.
\]
Additionally $m\tau_C$ is a scenario 2 time. Let $t_1$ be the first scenario 1 time after $m\tau_C$. So, by Lemma \ref{lemma:sync2} we have for all $t\in [m \tau_C , t_1]$,
\[
\|w(t)\|_{L^2}^2 \leq e^{-c_* (t- m\tau_C)} \|w(m\tau_C)\|_{L^2}^2 \leq e^{-c_* (t-m\tau_C) - c_*m\tau_C} \nu^2G^2 = e^{-c_* t }\nu^2G^2.
\]
We can repeat this argument to generate a sequence of times $t_i$ which grows without bound  so that 
\[
\|w(t_i)\|_{L^2}^2 \leq e^{-c_* t_i} \nu^2G^2.
\]

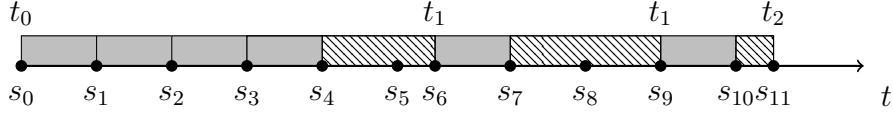
\begin{figure}\label{fig3}
\begin{center}
\begin{tikzpicture}
\draw (-5,0) -- (5,0);

\draw [fill=lightgray] (-5,0) -- (-5,.4) -- (-4,.4) -- (-4,0) -- (-5,0);

\draw [fill=lightgray] (-4,0) -- (-4,.4) -- (-3,.4) -- (-3,0) -- (-4,0);
\draw [fill=lightgray] (-3,0) -- (-3,.4) -- (-2,.4) -- (-2,0) -- (-3,0);
\draw [fill=lightgray] (-2,0) -- (-2,.4) -- (-1,.4) -- (-1,0) -- (-2,0);
\draw [fill=lightgray] (.5,0) -- (.5,.4) -- (1.5,.4) -- (1.5,0) -- (.5,0);
\draw [fill=lightgray] (3.5,0) -- (4.5,0) -- (4.5,.4) -- (3.5,.4) -- (3.5,0);

\draw (-2,.4) -- (.5,.4);
\draw (1.5,.4) -- (3.5,.4);
\draw (4.5,.4) -- (5,.4) -- (5,0);
 
 \draw[pattern=north west lines, pattern color=black] (-1,0) rectangle (.5,.4);
  \draw[pattern=north west lines, pattern color=black] (1.5,0) rectangle (3.5,.4);
   \draw[pattern=north west lines, pattern color=black] (4.5,0) rectangle (5,.4);

\filldraw[black] (-5,0) circle (2pt);
\filldraw[black] (-4,0) circle (2pt);
\filldraw[black] (-3,0) circle (2pt);
\filldraw[black] (-2,0) circle (2pt);
\filldraw[black] (-1,0) circle (2pt);
\filldraw[black] (0,0) circle (2pt);
\filldraw[black] (.5,0) circle (2pt);
\filldraw[black] (1.5,0) circle (2pt);
\filldraw[black] (2.5,0) circle (2pt);
\filldraw[black] (3.5,0) circle (2pt);
\filldraw[black] (4.5,0) circle (2pt);
\filldraw[black] (5,0) circle (2pt);

\draw (-5,.7) node {$t_0$};
\draw (.5,.7) node {$t_1$};
\draw (3.5,.7) node {$t_1$};
\draw (5,.7) node {$t_2$};

\draw (-5, -.4) node {$s_0$};
\draw (-4, -.4) node {$s_1$};
\draw (-3, -.4) node {$s_2$};
\draw (-2, -.4) node {$s_3$};
\draw (-1, -.4) node {$s_4$};
\draw (0, -.4) node {$s_5$};
\draw (.5, -.4) node {$s_6$};
\draw (1.5, -.4) node {$s_7$};
\draw (2.5,-.4) node {$s_8$};
\draw (3.5, -.4) node {$s_9$};
\draw (4.5, -.4) node {$s_{10}$};
\draw (5, -.4) node {$s_{11}$};


\draw[->,thick] (-5,0) -- (6.2,0);
\draw (6.5,-.4) node {$t$};

\end{tikzpicture}
\caption{
Refining the partition as in Figure \ref{fig2}. The new partition is segmented using the points labelled $s_i$. Note that $|s_{i+1}-s_i|\leq \tau_C$.
}
\end{center}
\end{figure}

This decay only holds on a sequence of times. We now extend this to decay for all times $t\geq \tau_C$. We can define a second sequence $s_i$ so that $|s_{i+1}-s_{i}|\leq \tau_C$ and $\{ t_i\}\subset \{s_i\}$, see Figure \ref{fig3}. This sequence also clearly satisfies $\| w(s_i)\|_{L^2}^2 \lesssim e^{-c_*s_i}\nu^2G^2$.
Then for $t\in (s_{i} , s_{i+1})$ we have by \eqref{ineq:EnergyBoundw},
\EQN{
\| w(t)\|_{L^2}^2 &\leq  e^{C_L^4\nu\la_1 G^2(t-s_i)} \|w(s_{i}) \|_{L^2}^2
\\&\leq   e^{C_L^4\nu\la_1 G^2\tau_C} \|w(s_{i}) \|_{L^2}^2
\\&\leq   e^{C_L^4\nu\la_1 G^2\tau_C - c_*s_i} \nu^2G^2.
}
Note that for $\tau_C\leq s_i$ we have $t/2 \leq s_i$.  So,
 \[
\| w(t)\|_{L^2}^2 \leq    e^{C_L^4\nu\la_1 G^2\tau_C - c_*t/2} \nu^2G^2 \leq e^{C_L^4} \nu^2G^2 e^{-c_*t/2},
\]
where we used the fact that $\tau_C\leq \mu^{-1}\leq  (
\nu \la_1 G^2)$ (which is \eqref{C4}).

The condition for $\mu$ in the theorem's statement is obtained from \eqref{C4} and \eqref{C9} while that for $h$ is from \eqref{C1}  and \eqref{C2}. Presently, $\tau_C$ has been defined to be the smallest quantity in the right hand sides of \eqref{C5}, \eqref{C6} and \eqref{C7}. In the theorem's statement, however, we reduce this value for the sake of readability (it is still based on \eqref{C5}, \eqref{C6} and \eqref{C7}).

\end{proof}

\section{Numerical tests}

Our tests are carried out on the NSE in vorticity form
\begin{align}\label{NSEom}
    \frac d {dt} \omega -\nu\Delta \omega + u \cdot \nabla \omega = g\;,\quad  u =\nabla^{\perp}\psi,\quad -\Delta \psi =\omega
\end{align}
where $g=\nabla \times f$, the same time independent force concentrated on the annulus with wave numbers  $10\le |\bk| < 12$, as used in \cite{OT1,OT2, BBJ1}. The analysis has been done in terms of velocity, yet if it is robust, we should in practice see similar effects for vorticity.  
We solve both \eqref{NSEom} and the nudged equation
\begin{align}\label{NSEtom}
    \frac d {dt} \tom -\nu\Delta \tom + \tu \cdot \nabla \tom = g - J_h (\tom-\om) ;,\quad  \tu =\nabla^{\perp}\tpsi,\quad -\Delta \tpsi =\tom
\end{align}
using a fully dealiased pseudospectral code with N=512 modes in each direction over the full physical domain $[0,2\pi]^2$. 

As in \cite{BBJ1}, for all mobile local nudging we use a spectrally filtered interpolating operator $J_h$. We first move from Fourier coefficients to nodal values via an FFT$^{-1}$ applied to $\tom-\om$.  We then use data at only every  $2^p$-th node in each direction, so that $p=1,2,3,4$, corresponds to $h=\pi/128, \pi/64, \pi/32$ and $\pi/16$, respectively. A recursively averaged operator $\interp_p$ depicted in Figure \ref{abcdfig} is used the smoothen the result over the subdomain $D$.   After restricting to the subdomain, we transform back so that
\begin{align*}\label{interpt}
J_h(\tilde\omega-\omega)=\text{FFT} \circ\interp_p\circ \chi_{\tiny D} \circ\text{FFT}^{-1}(\tilde\omega-\omega)\;.
\end{align*}

The viscosity is fixed at $\nu=10^{-4}$, and the force is scaled so that the {(traditional) Grashof number is $$\frac{\|f\|_{L^2}}{\nu^2 \lambda_1}=10^6\;.$$}This results in a chaotic reference solution, which after 25,000 time units starting from zero initial vorticity, is presumed to essentially be on the global attractor.
The time stepper is the third-order Adams-Bashforth method in \cite{OT1,OT2} in which the linear term is handled exactly through an integrating factor.  Unless specified otherwise, the step size is taken to be $\delta t=0.001$.  A larger step $\delta t=0.01$ was found in \cite{OT2} to be sufficient for computing the reference solution, but the new nudging schemes synchronize quickly if the flow is sampled on a finer time scale.  The relaxation parameter is fixed at $\mu=50$ and nudging takes place over a moving subdomain of size $\pi/2 \times \pi/2$.
 
\begin{figure}[ht]
\psfrag{a}{\tiny$a$}
\psfrag{b}{\tiny$b$}
\psfrag{c}{\tiny$c$}
\psfrag{d}{\tiny$d$}
\psfrag{ab}{\tiny$\frac{a+b}{2}$}
\psfrag{bc}{\tiny$\frac{b+c}{2}$}
\psfrag{cd}{\tiny$\frac{c+d}{2}$}
\psfrag{ad}{\tiny$\frac{a+d}{2}$}
\psfrag{abcd}{\tiny$\frac{a+b+c+d}{4}$}
 \centerline{\includegraphics[scale=.5 ]{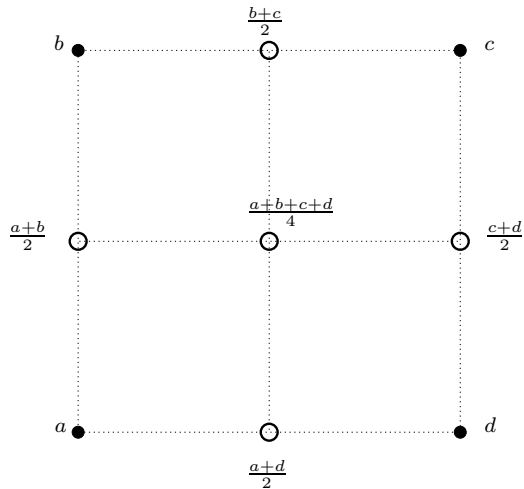}}
\caption{First recursive step of $\interp_p$.  Values of $\tilde\omega_N-\omega_N$ are $a$, $b$, $c$, $d$ at the corners.}
\label{abcdfig}
\end{figure}

\subsection{Periodic movement of nudging subdomain}

We consider two similar movements of the subdomain where the nudging takes place. The lower left corner of the subdomain is determined by periodic functions $n_x(\tau)$, $n_y(\tau)$. In one scheme, these functions are as in \cite{BBJ1}, and depicted in Figure \ref{cornerfig2} (Top). The subdomain would move continuously (if not for discrete time steps), except at the end of the cycle, when it jumps down to the lower left corner. For this reason, we call such schemes {\it discontinuous periodic}. In contrast, the functions $n_x(\tau)$, $n_y(\tau)$ in Figure \ref{cornerfig2} (Bottom) result in a {\it continuous periodic} scheme.

Speed of movement for both schemes is adjusted by varying the frequency $F$ of the function
$\tau(t)$.  Due to the difference in periods, in the discontinuous case
$\tau(t)=15(Ft-\lfloor Ft\rfloor)$, while in the continuous case  $\tau(t)=16(Ft-\lfloor Ft\rfloor)$. Thus if $F=1$, the subdomain would complete the cycle in one time unit.

\begin{figure}[ht]
\psfrag{3}{\tiny$3$}
\psfrag{4}{\tiny$4$}
\psfrag{7}{\tiny$7$}
\psfrag{8}{\tiny$8$}
\psfrag{11}{\tiny$11$}
\psfrag{12}{\tiny$12$}
\psfrag{15}{\tiny$15$}
\psfrag{6}{\tiny$6$}
\psfrag{13}{\tiny$12$}
\psfrag{10}{\tiny$10$}
\psfrag{16}{\tiny$16$}
\psfrag{9}{\tiny$9$}
\psfrag{ix}{\tiny$n_x(\tau)$}
\psfrag{iy}{\tiny$n_y(\tau)$}
\psfrag{x}{\tiny$x$}
\psfrag{y}{\tiny$y$}
\psfrag{t}{\tiny$\tau$}
\psfrag{N2}{\tiny$N/2$}
\psfrag{N4}{\tiny$N/4$}
\psfrag{N3}{\tiny$3N/4$}
 \centerline{\includegraphics[scale=.5 ]{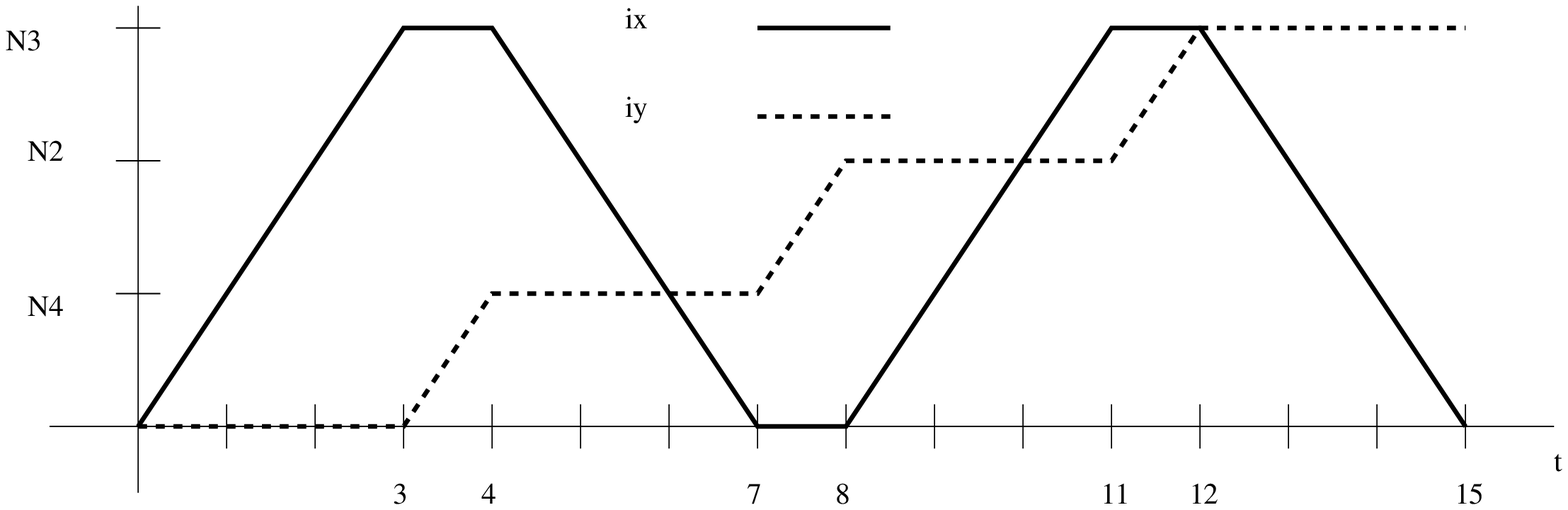}\quad \includegraphics[scale=.35 ]{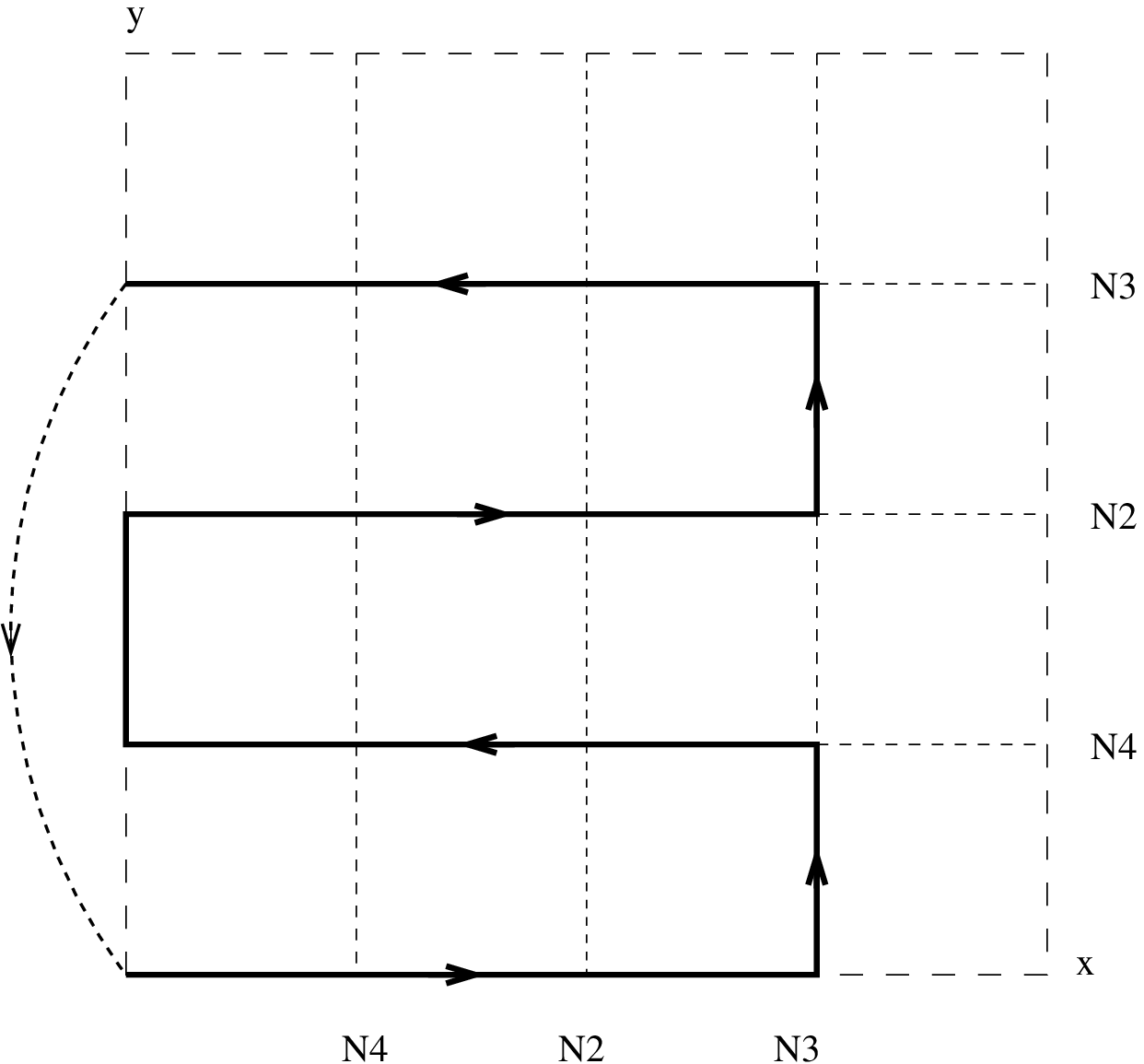} }
 \centerline{\includegraphics[scale=.5 ]{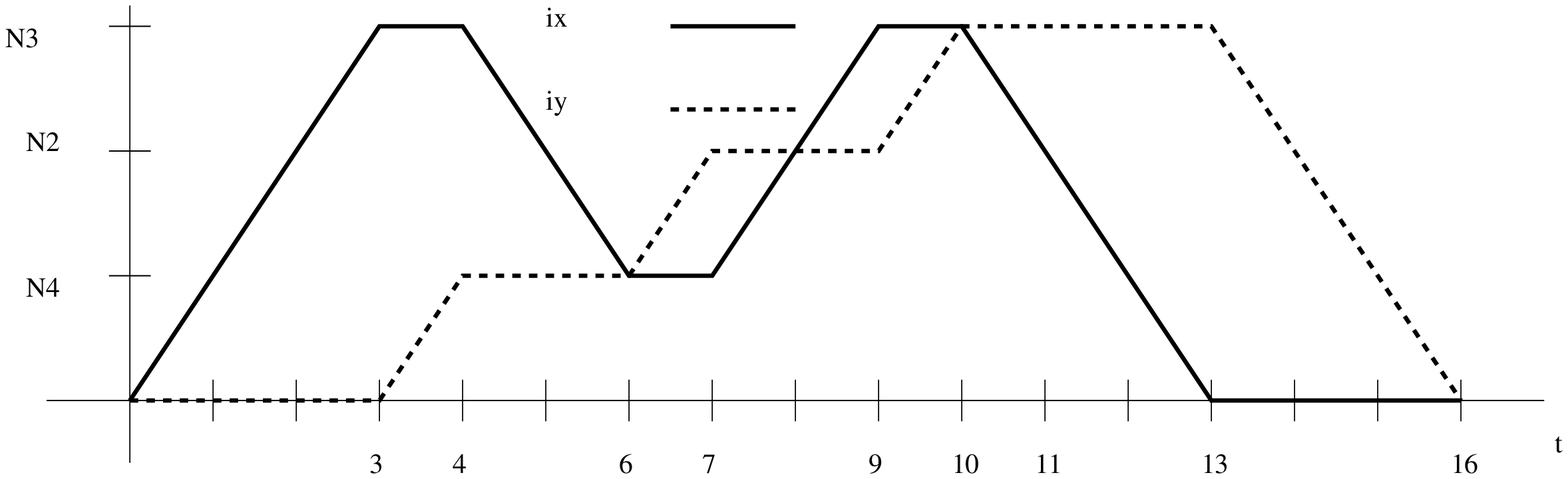}\quad \includegraphics[scale=.35 ]{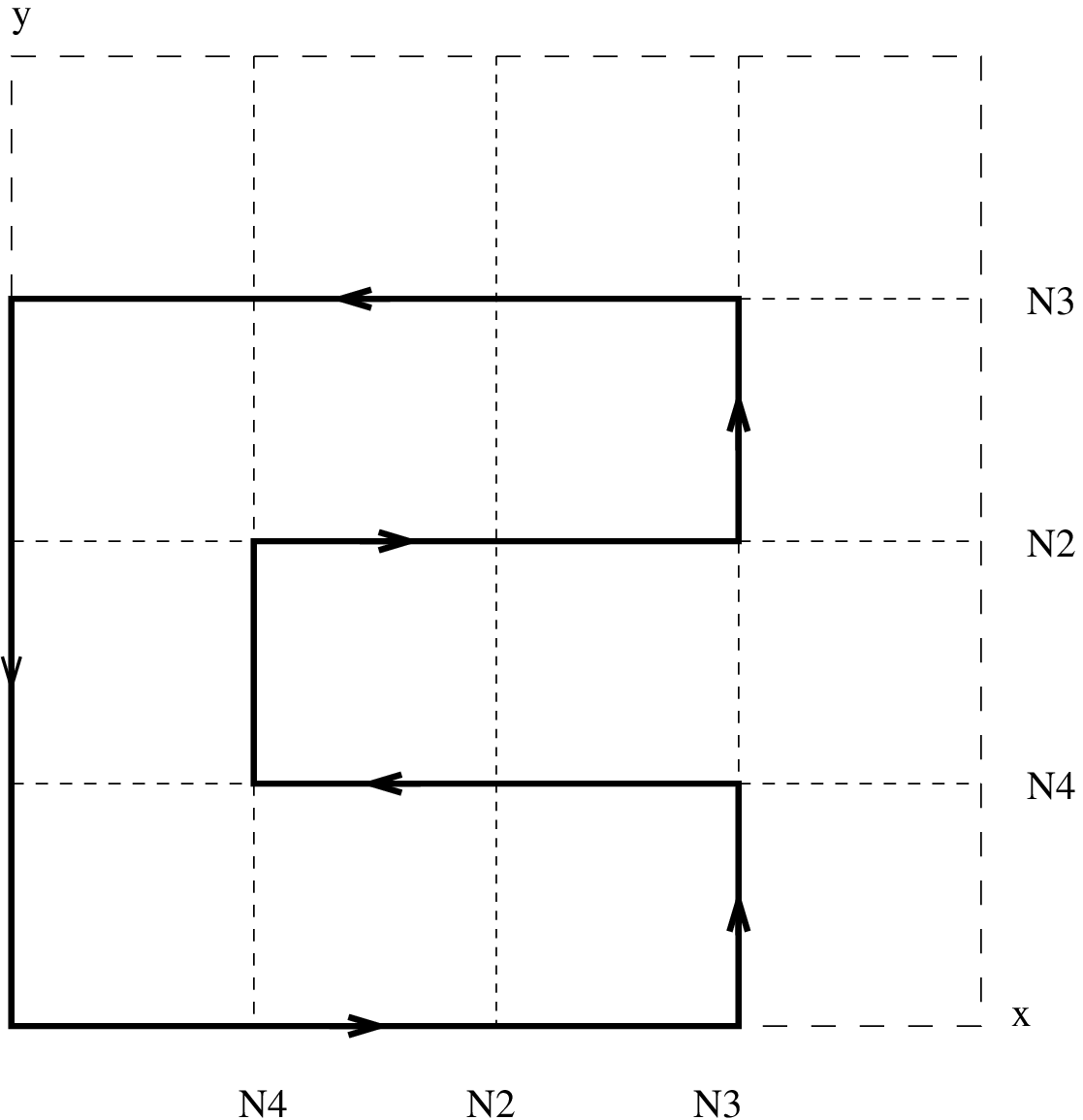} }
 \caption{Movement of lower left corner of subdomain. Top: discontinuous scheme, Bottom: continuous scheme.}
\label{cornerfig2}
\end{figure}

In Figure \ref{varyspeedfig} (Left) the relative $L^2$ error using the discontinuous periodic scheme is plotted for $F$ ranging from 0.25 to 40 (the legend from top to bottom matches the ranking of the errors at $t=100$).  In each case we use the interpolating operator $\interp_4$ over the subdomain, i.e., only every 16th nodal value in each direction.  While Figure \ref{varyspeedfig} shows nudging with the moving subdomain effectively synchronizes with most choices of $F$, there appears to be sweet spot at $F=0.5$.  

We note that nudging at frequencies $F=8,16,40$ does not result in synchronization.  This seems to be due to a resonance that causes only a small fraction of the possible subdomain positions to be visited.   This is illustrated in Figure \ref{varyspeedfig} (Right) where we plot the corners $n_x,n_y$ of the subdomains through the first two cycles for $F=40$ and $\delta t=0.001$.   This resonance can be avoided by reducing the time step size.  Figure \ref{move40fig} (Left) shows that the result for $F=40$, $\delta=0.0001$ nearly matches that for  $F=4$ and $\delta t=.001$. The eventual discrepancy between the two error plots is presumably due to increased round-off from a ten-fold increase in computational steps. 
The smaller time step prevents skipping over too many subdomains, as shown in Figure \ref{move40fig} (Right).

\begin{figure}[ht]
\psfrag{t}{\tiny$t$}
\psfrag{1over16}{\tiny$1/16$}
\psfrag{relative L2error}{\tiny $L^2$ rel. error}
\psfrag{'nxE-3' u 3:1}{\qquad\tiny$n_x$}
\psfrag{'nyE-3' u 3:2}{\qquad\tiny$n_y$}
\centerline{\includegraphics[scale=.4,angle=-90]{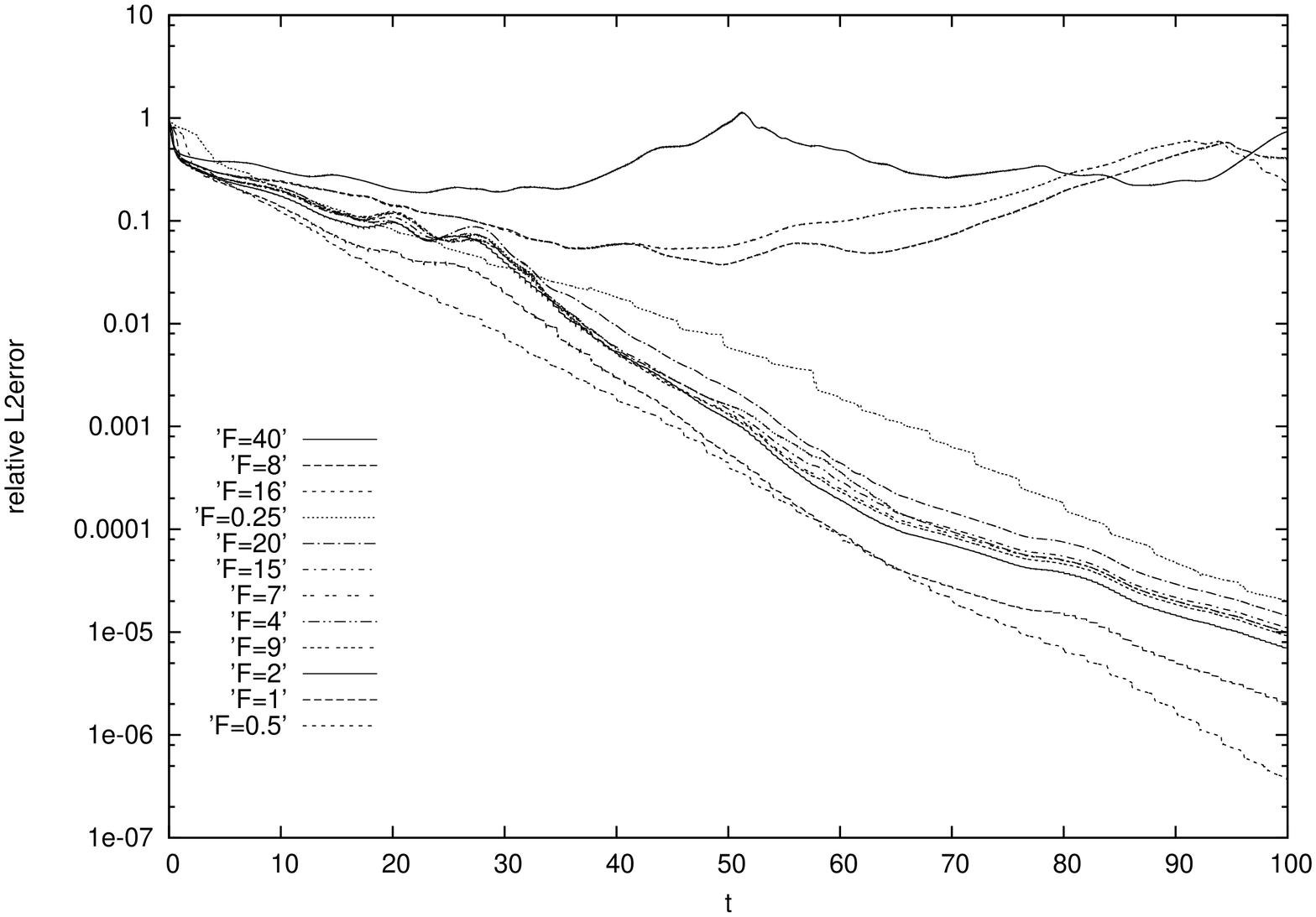} \includegraphics[scale=.4,angle=-90]{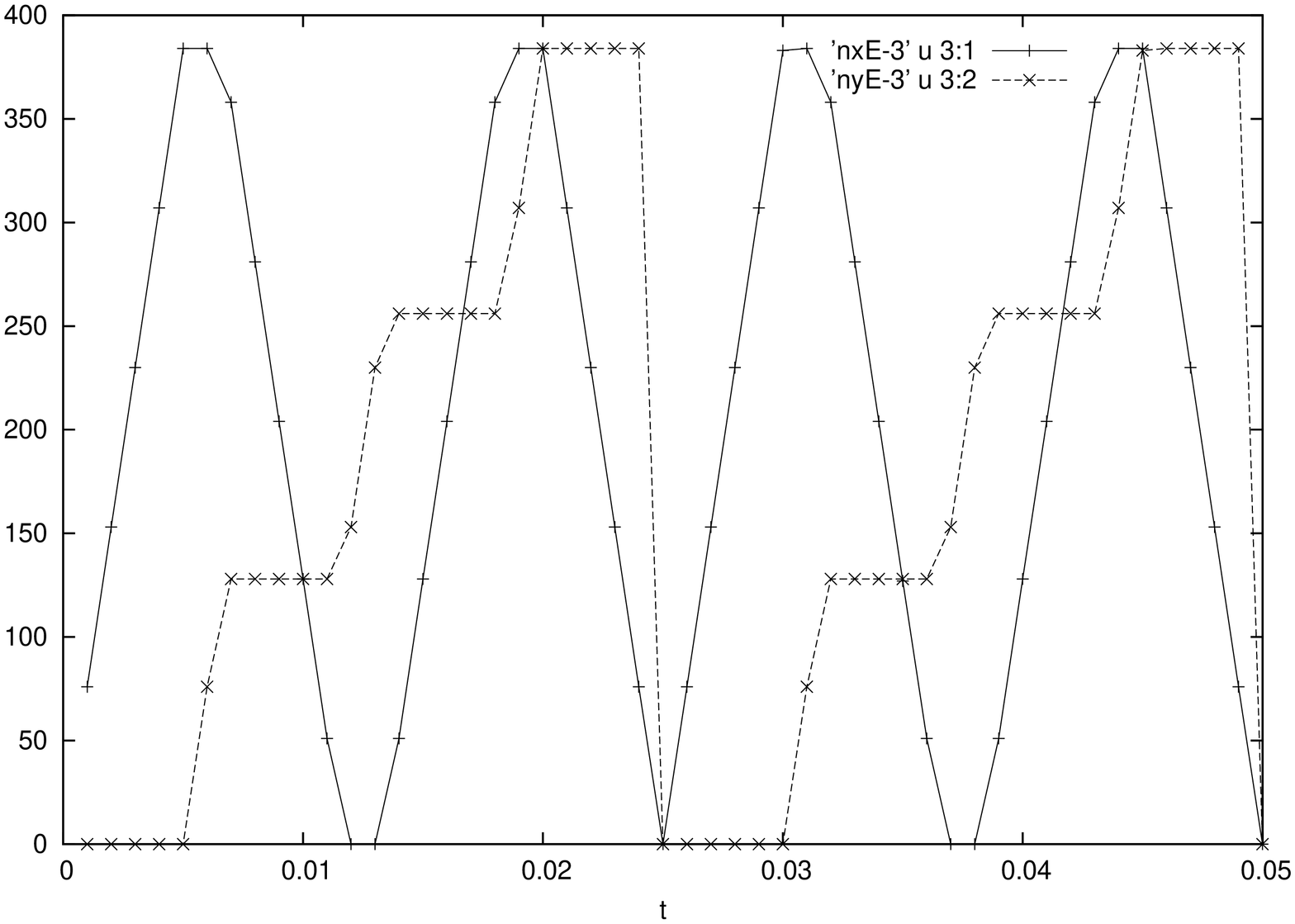}  }
\vskip .15 truein
\caption{Left: effect of frequency, $F$, $\delta t=0.001$.  Right: subdomain corners $n_x,n_y$, $F=40$, $\delta t=0.001$.  }
\label{varyspeedfig}
\end{figure}

\begin{figure}[ht]
\psfrag{t}{$t$}
\psfrag{'nxE-4' u 3:1}{\qquad\tiny$n_x$}
\psfrag{'nyE-4' u 3:2}{\qquad\tiny$n_y$}
\psfrag{relative L2error}{\tiny $L^2$ rel. error}
\centerline{
 \includegraphics[scale=.4,angle=-90]{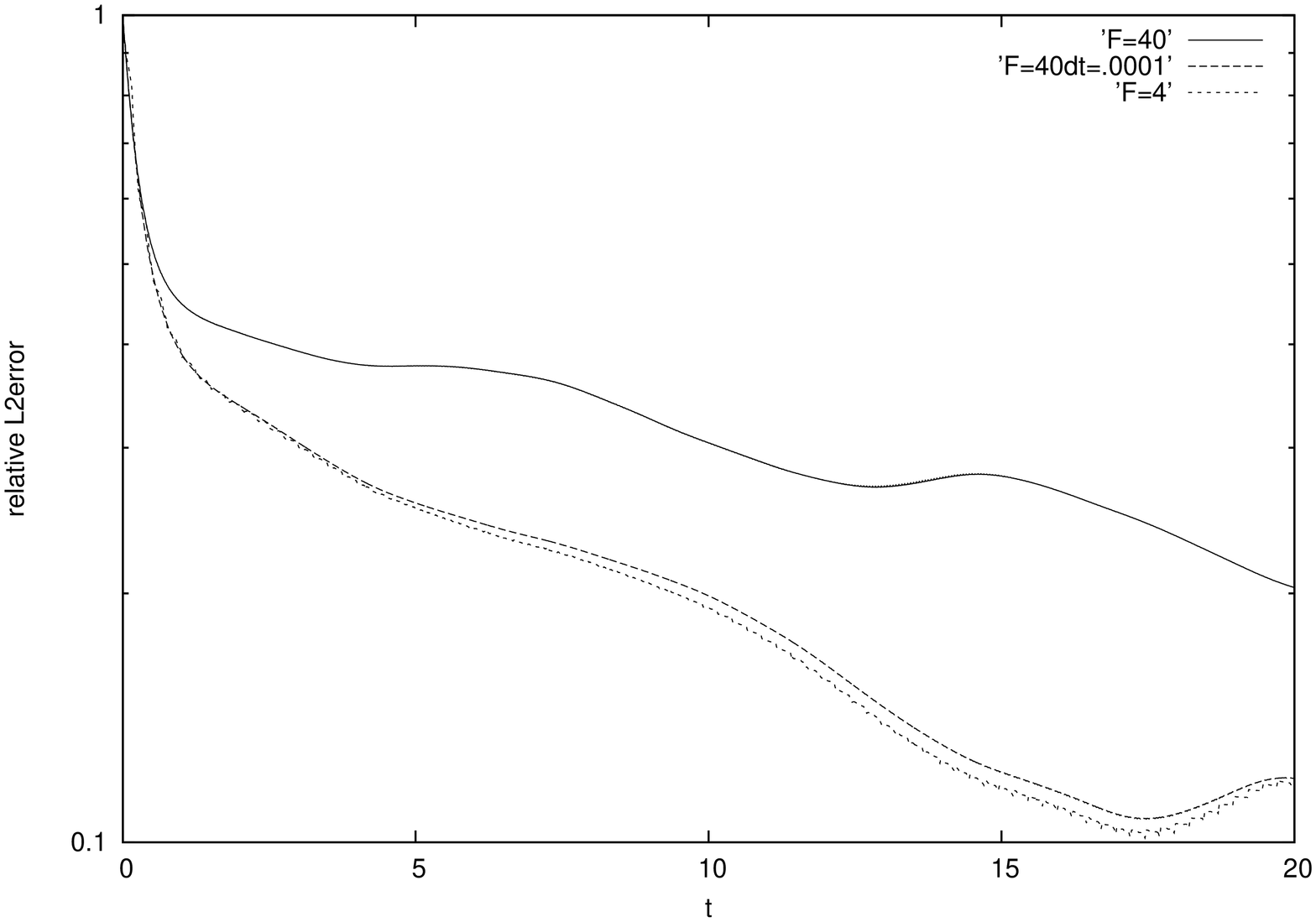}
 \includegraphics[scale=.4,angle=-90]{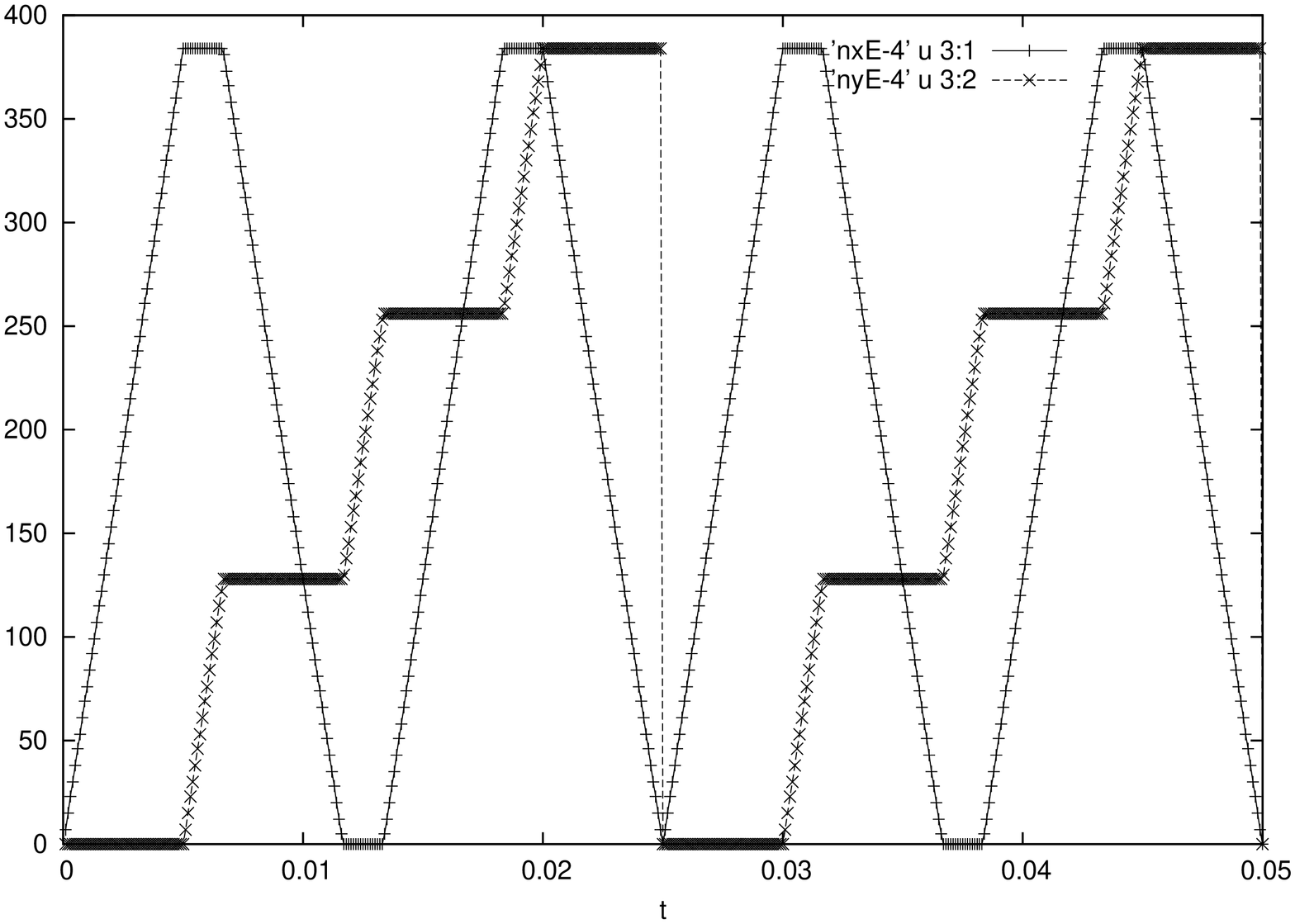} }
\vskip .15 truein
\caption{Left: $F=4$ with $\delta t=0.001$, $F=40$ with both $\delta t =0.001$ and  $\delta t =.0001$. Right: subdomain corners $n_x,n_y$, $F=40$, $\delta t=0.0001$.}
\label{move40fig}
\end{figure}

In \cite{BBJ1} we aimed to use the minimal amount of data. We compare the effect of data resolution on the continous scheme in Figure \ref{p1ctsfig} (Left).  The results for $p=1$ and $p=2$ are indistinguishable.

\begin{figure}[ht]
\psfrag{t}{$t$}
\psfrag{'../T=.02end'}{\quad\tiny{dom.}}
\psfrag{'p1ctsF1'}{\tiny $p=1$}
\psfrag{'p2ctsF1'}{\tiny $p=2$}
\psfrag{'p3ctsF1'}{\tiny $p=3$}
\psfrag{'p4ctsF1'}{\tiny $p=4$}
\psfrag{'p1cts'}{\tiny{cts}}
\psfrag{'p1dct'}{\tiny disc }
\psfrag{relative L2error}{$L^2$ rel. error}
\centerline{\includegraphics[scale=.4,angle=-90]{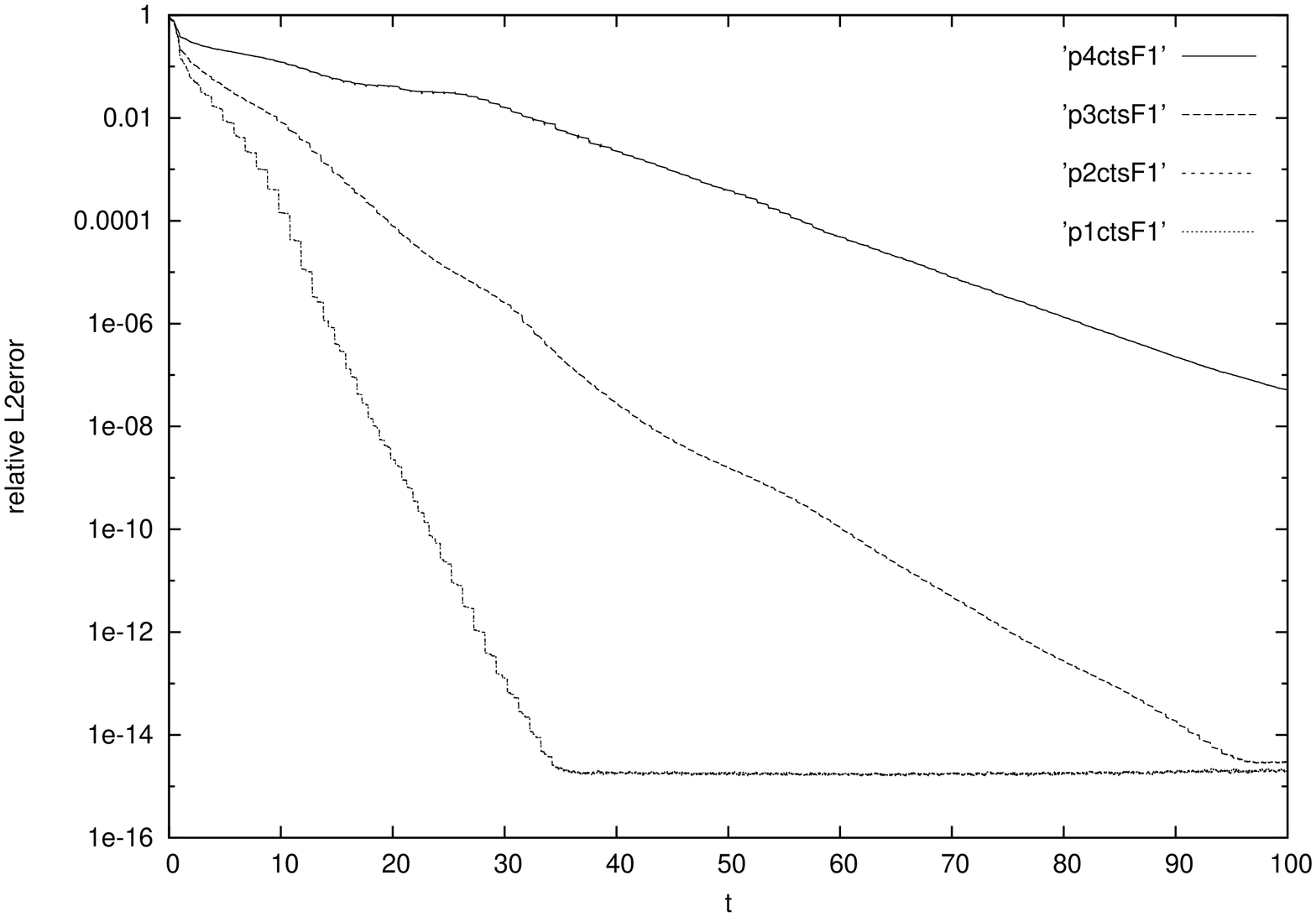}  \includegraphics[scale=.4,angle=-90]{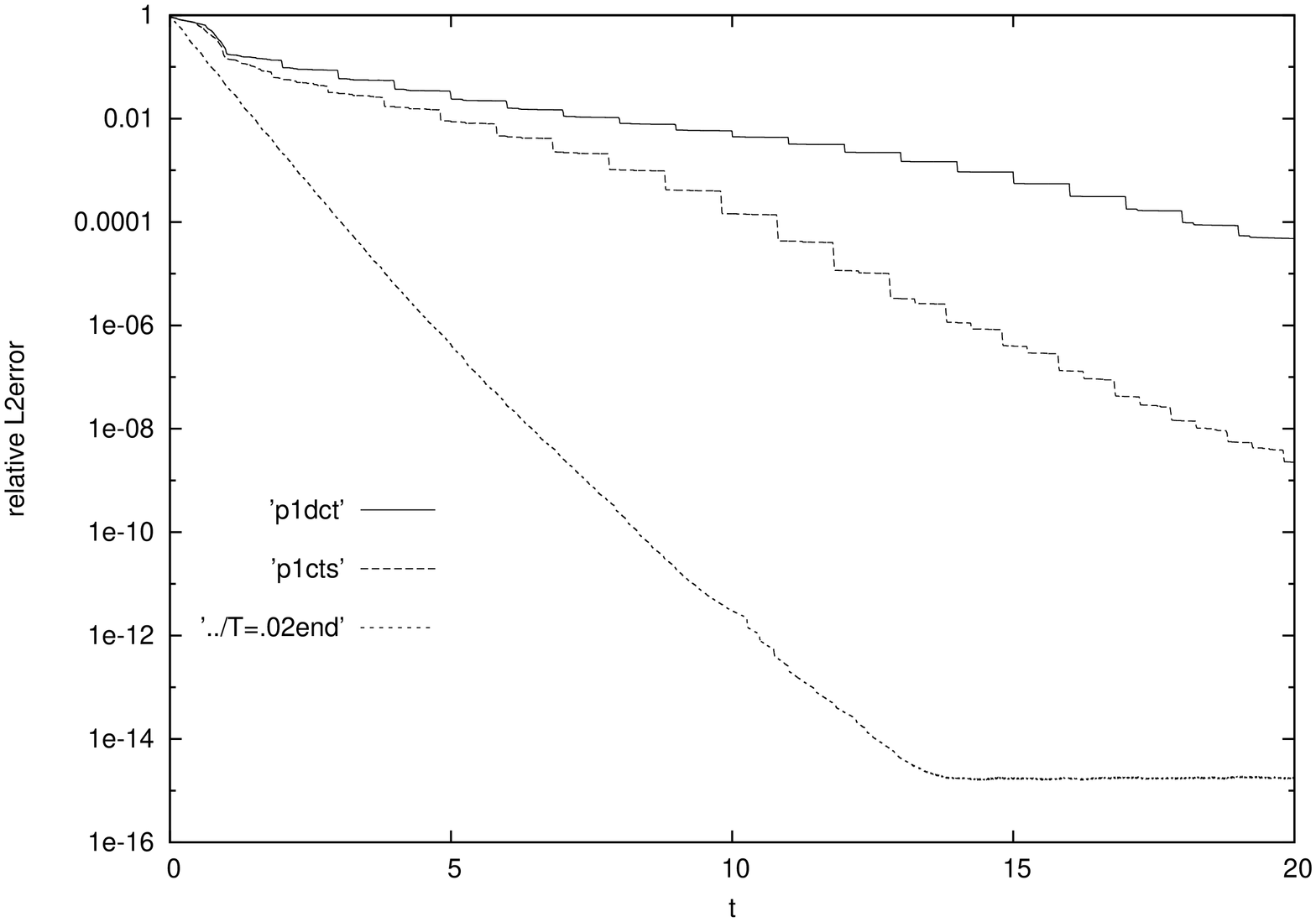} }
\vskip .15 truein
\caption{Left: effect of data resolution on continuous scheme, $F=1$. Right: continuous and discontinuous movements, including the dominant scheme with nudging period $T=.02$, all with $p=1$.}
\label{p1ctsfig}
\end{figure}

\subsection{Nudging over the dominant subregion}

Motivated by the analysis in this paper, we also test a scheme which finds among the 16 subdomains in Figure \ref{cornerfig2} (Right) the dominant one, nudges for a fixed period $T$ over that subdomain, and repeats.  The nudging on the subdomain is done using every other node in each direction, i.e., $p=1$. The dominant subdomain is determined by computing the  trapezoidal rule approximation of the integral $\int_{\Omega_j} |\omega|^2 $ for $j=1, \ldots, 16$ on a coarse grid: using every $16^{\text{th}}$ node in each direction. In Figure \ref{p1ctsfig} (Right) we compare all three movement schemes, using in each case, every other data point within the subdomain.  The scheme nudging over the dominant subdomain achieves near machine precision in about half the time it took for the continuous scheme to do so.  The relative $L^2$ errors from the reference solution are shown in Figure \ref{up100fig} (Left) for several values of $T$.  A zoom of those errors over the initial time range is shown in Figure \ref{up100fig} (Right), together with the ratios of  $R_T=\int_{\Omega_j} |\omega|^2/\int_{\Omega} |\omega|^2$ for $T=.04,.08$ and the line for the constant function $1/16$.  As expected, the error drops the fastest when when this ratio is largest. We note that the number of observers, when using the nudging period $T=.02$, $p=1$, and $\delta t=.001$, comes to $64^2+32^2/20=
4147.2$ per step.

\begin{figure}[ht]
\psfrag{t}{$t$}
\psfrag{relative L2error}{$L^2$ rel. error}
\psfrag{'T=.08end'}{\tiny$T=.08$}
\psfrag{'T=.04end'}{\tiny$T=.04$}.
\psfrag{'T=.02end'}{\tiny$T=.02$}
\psfrag{'T=.04ratio'}{\tiny$R_{.04}$}
\psfrag{'T=.08ratio'}{\tiny$R_{.08}$}
\psfrag{'1over16'}{\tiny$1/{16}$}
\centerline{\includegraphics[scale=.4,angle=-90]{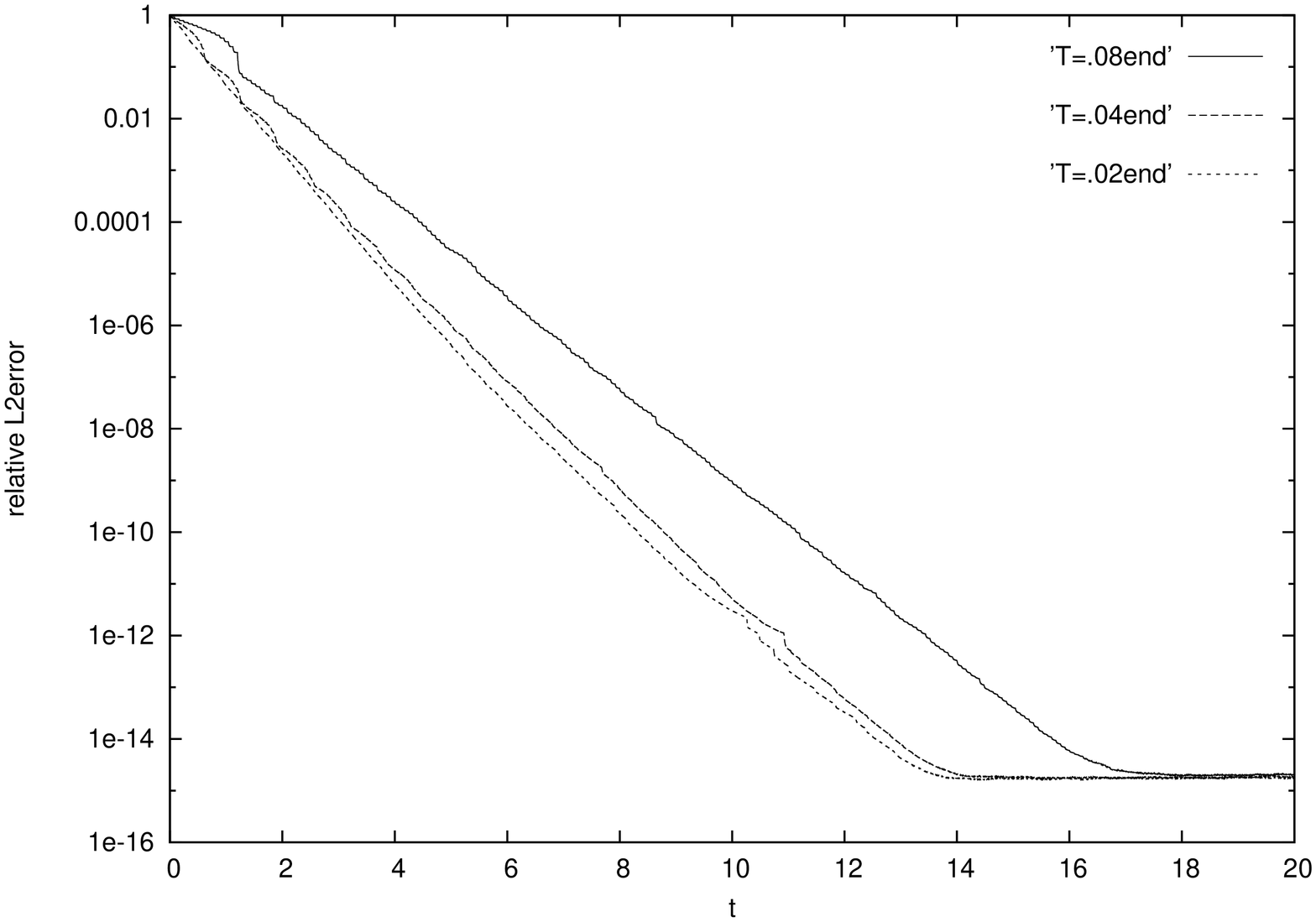}  \includegraphics[scale=.4,angle=-90]{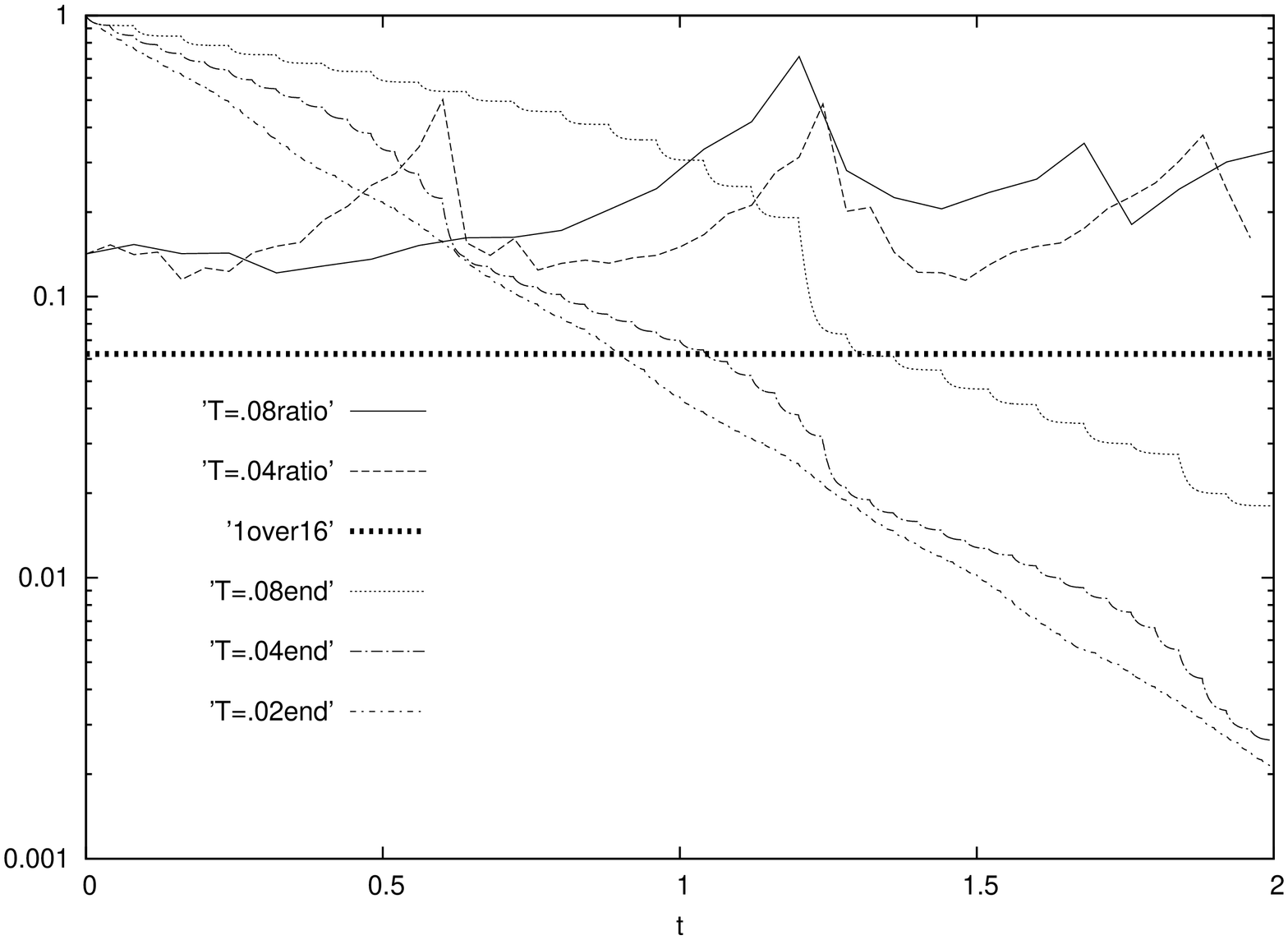} }
\vskip .15 truein
\caption{Left: effect of update period on dominant scheme. Right: zoom of errors over initial time range along with their ratios $R_T$ for $T=.04, .08$. }
\label{up100fig}
\end{figure}

One can imagine a scenario where some time would be needed to move observational equipment from one subdomain to the next. We simulate this by introducing a delay at the moment the new dominant subdomain is determined.  The delay is taken to be a significant fraction $\varphi_{\text d}$ of the nudging period: in one case, one-fourth and in another, one-half.  With the delay being at the beginning of that period, it is natural to expect a slowdown in synchronization, since the dominance of the subdomain fades.  The results for these delays are shown in Figure \ref{delayfig} for the period $T=.02$.  Also shown in Figure \ref{delayfig}  is a plot for a scheme where the next subdomain is chosen randomly, with no delay as well as with delay fractions $\varphi_{\text r}=.25, .5$ of the nudging period. Such random movement is somewhat akin to the ``bleeps" scheme in \cite{LariosMobile}, except here the observers remain fixed and uniformly distributed over a subdomain for the nudging period, rather than moving freely throughout the full domain.  
\begin{figure}[ht]
\psfrag{t}{$t$}
\psfrag{relative L2error}{$L^2$ rel. error}
\psfrag{'frac=0.500'}{\tiny $\varphi_{\text d}=.5$}
\psfrag{'random.25'}{\tiny $\varphi_{\text r}=.25$}
\psfrag{'frac=.2500'}{\tiny  $\varphi_{\text d}=.25$}
\psfrag{'random0.5'}{\tiny $\varphi_{\text r}=.5$}
\psfrag{'../T=.02end'}{\quad\tiny dom.}
\psfrag{'random000'}{\tiny random}
\centerline{\includegraphics[scale=.4,angle=-90]{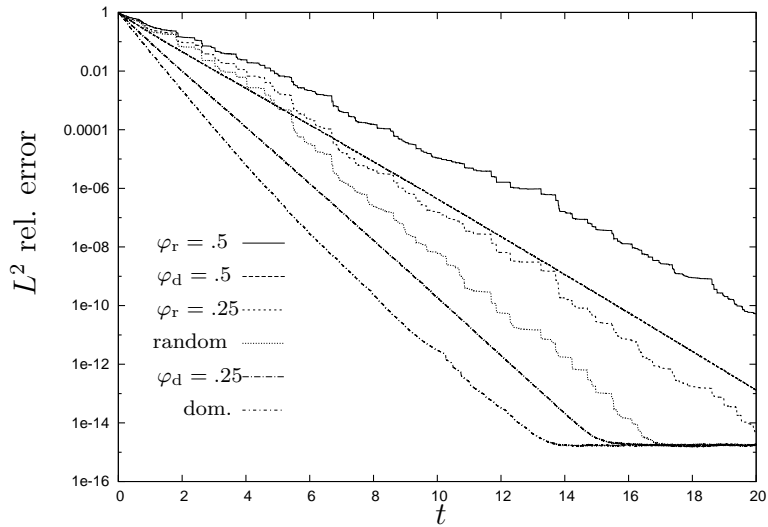} }
\vskip .15 truein
\caption{Effect of delay in nudging, $T=.02$, $\varphi_{\text d}=$ delay fraction dominant scheme, $\varphi_{\text r}=$ delay fraction random scheme. }
\label{delayfig}
\end{figure}

In our final test,  we compare to purely spectral nudging.  Recall that the data we have used to determine the dominating subregion corresponds, through an FFT, to the $32^2$ lowest Fourier modes.  Nudging at this global resolution can be simply affected by taking the interpolating operator to be projection onto those modes.  We plot the result in Figure \ref{hybridfig}, together with that for the dominant scheme for period $T=.02$ (no delay) and that for spectral nudging with $64^2=4096$ modes (nearly the amount of data, per time step, used for the dominant scheme).   We note that, initially, the error decreases faster for spectral nudging.  To take advantage of this, we consider a 
hybrid scheme which uses spectral nudging with $32^2$ modes until time $t=.1$ and then switches to the dominant scheme.  This results in faster synchronization to machine precision, though it used knowledge of an optimal time to switch. Such a hybrid scheme could be made practical by monitoring the error as Fourier projection is used and switching when a significant slowdown is detected.    

\begin{figure}[ht]
\psfrag{t}{$t$}
\psfrag{'32sqmodes'}{\tiny$32^2$ modes}
\psfrag{'64sqmodes'}{\tiny$64^2$ modes}
\psfrag{'hybrid'}{\tiny hybrid}
\psfrag{'T=.02end'}{\tiny dom.}
\psfrag{relative L2error}{$L^2$ rel. error}
\centerline{\includegraphics[scale=.4,angle=-90]{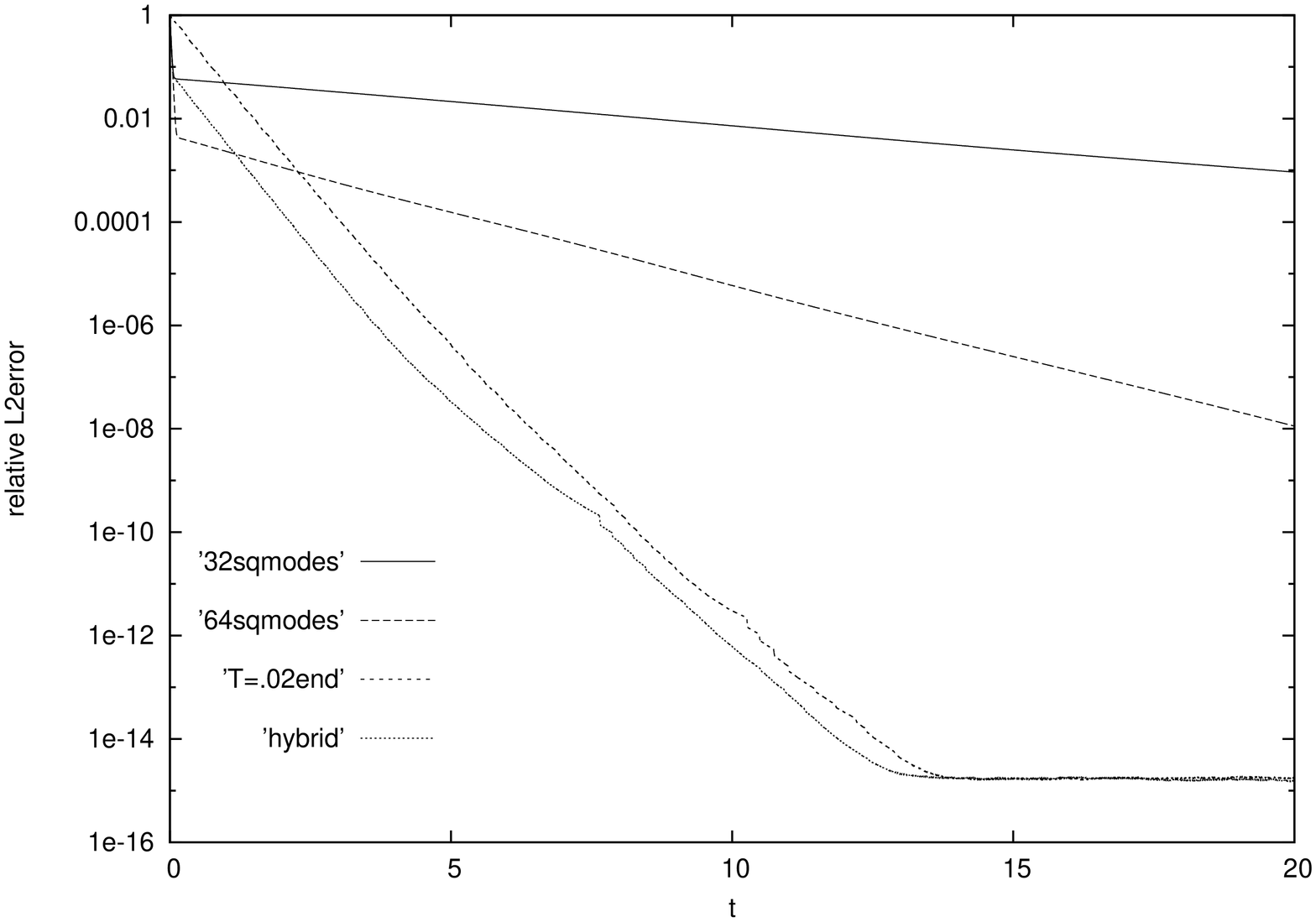}\includegraphics[scale=.4,angle=-90]{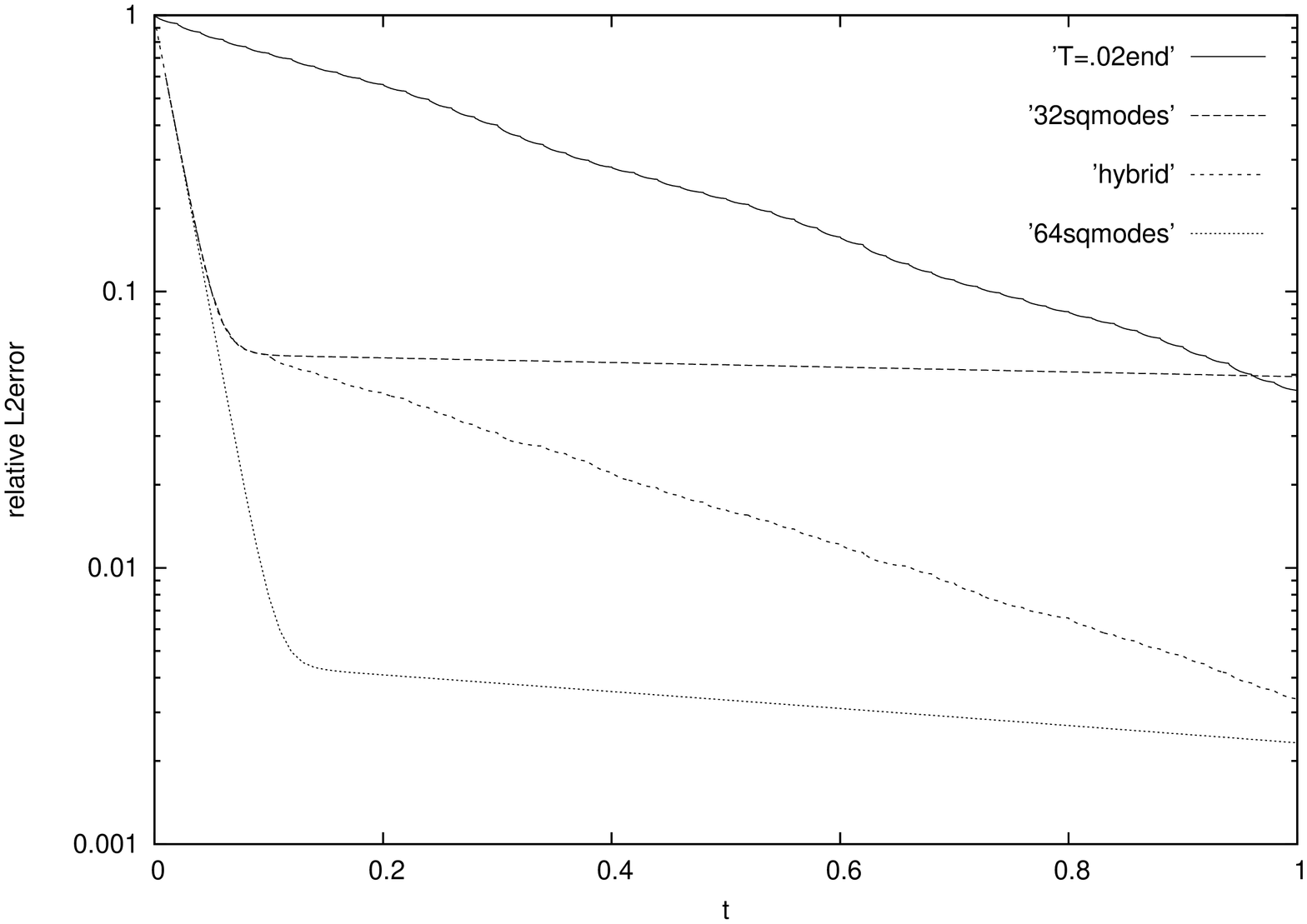}}
\vskip .15 truein
\caption{Comparing spectral nudging with dominant scheme and a hybrid scheme, $T=.02$, switch at $t=.1$, all done with $\delta t = 0.001$.}
\label{hybridfig}
\end{figure}

\subsection{Summary of numerical results}

We have demonstrated that when the nudging subdomain moves in a regular pattern to cover the full domain, generally, the faster it moves, the faster the synchronization. Refining the data mesh over the subdomain from every 16th point down to every 4th point, in each direction, yields synchronization to machine precision in less than 35 time units. Nudging with every other point in each direction does not change the convergence for regular continuous movement. Nudging at that same resolution over the subdomain with the dominant $L^2$ error achieves machine precision in about half the time.  To be fair, this dominant scheme requires an additional coarse mesh of observed data in order to determine the subdomain. This, however, is done only at the end of each nudging period. We have found a near optimal time period, over which to do the nudging before the determination of the next dominant subdomain. This dominant scheme synchronizes somewhat faster than nudging for the same period with a subdomain that is randomly selected. In the case of the random scheme, over certain time intervals the error is nearly constant, punctuated by sharp drops when, presumably, the subdomain is dominant, or nearly so.  These sharp declines are also to be expected since over the prior interval the error has held roughly steady, making the beneficial feedback stronger.  Both the dominant and random scheme appear to be robust against a significant delay before starting to nudge after each new subdomain is determined.  Finally, we have found that while the dominant scheme achieves much faster synchronization with a comparable number of data observations (per time step) than using global spectral nudging, the latter produces a sharper initial drop.  We have demonstrated the effectiveness of a simple hybrid scheme that switches from the spectral scheme to dominant scheme, after this initial drop.

 \section*{Acknowledgements} 
 Bradshaw acknowledges support from the Simons Foundation via a Collaboration Grant (635438). The authors all acknowledge the Lilly Endowment, Inc., through its support for the Indiana University Pervasive Technology Institute, which provided supercomputing resources used for this research \cite{IUcomp}.

\end{document}